\documentclass[9pt]{amsart}
\usepackage{amsmath}
\usepackage{amsxtra}
\usepackage{amscd}
\usepackage{amsthm}
\usepackage{amsfonts}
\usepackage{amssymb}
\usepackage{eucal}
\usepackage[all]{xy}
\usepackage{graphicx}

\newtheorem{cor}[subsubsection]{Corollary}

\newtheorem{lem}[subsubsection]{Lemma}
\newtheorem{prop}[subsubsection]{Proposition}

\newtheorem{conj}[subsubsection]{Conjecture}
\newtheorem{thm}[subsubsection]{Theorem}

\theoremstyle{definition}

\theoremstyle{remark}
\newtheorem{rem}[subsubsection]{Remark}

\newcommand{\thmref}[1]{Theorem~\ref{#1}}

\newcommand{\secref}[1]{Sect.~\ref{#1}}
\newcommand{\lemref}[1]{Lemma~\ref{#1}}
\newcommand{\propref}[1]{Proposition~\ref{#1}}
\newcommand{\corref}[1]{Corollary~\ref{#1}}
\newcommand{\conjref}[1]{Conjecture~\ref{#1}}

\numberwithin{equation}{section}

\newcommand{\nc}{\newcommand}
\nc{\renc}{\renewcommand}
\nc{\ssec}{\subsection}
\nc{\sssec}{\subsubsection}
\nc{\on}{\operatorname}

\nc\ol{\overline}
\nc\wt{\widetilde}
\nc\tboxtimes{\wt{\boxtimes}}
\nc\tstar{\wt{\star}}
\nc{\alp}{a}

\nc{\ZZ}{{\mathbb Z}}
\nc{\NN}{{\mathbb N}}
\nc{\OO}{{\mathbb O}}
\renc{\SS}{{\mathbb S}}
\nc{\DD}{{\mathbb D}}
\nc{\GG}{{\mathbb G}}

\nc{\Fq}{{\mathbb F}_q}
\nc{\Fqb}{\ol{{\mathbb F}_q}}
\nc{\Ql}{\ol{{\mathbb Q}_\ell}}
\nc{\id}{\text{id}}
\nc\X{\mathcal X}

\nc{\Hom}{\on{Hom}}
\nc{\Lie}{\on{Lie}}
\nc{\Loc}{\on{Loc}}
\nc{\Pic}{\on{Pic}}
\nc{\Bun}{\on{Bun}}
\nc{\IC}{\on{IC}}
\nc{\Aut}{\on{Aut}}
\nc{\rk}{\on{rk}}
\nc{\Sh}{\on{Sh}}
\nc{\Perv}{\on{Perv}}
\nc{\pos}{{\on{pos}}}
\nc{\Conv}{\on{Conv}}
\nc{\Sph}{\on{Sph}}
\nc{\Sym}{\on{Sym}}
\nc{\BunBb}{\overline{\Bun}_B}
\nc{\BunNb}{\overline{\Bun}_N}
\nc{\BunTb}{\overline{\Bun}_T}
\nc{\BunBbm}{\overline{\Bun}_{B^-}}
\nc{\BunBbel}{\overline{\Bun}_{B,el}}
\nc{\BunBbmel}{\overline{\Bun}_{B^-,el}}
\nc{\Buno}{\overset{o}{\Bun}}
\nc{\BunPb}{{\overline{\Bun}_P}}
\nc{\BunBM}{\Bun_{B(M)}}
\nc{\BunBMb}{\overline{\Bun}_{B(M)}}
\nc{\BunPbw}{{\widetilde{\Bun}_P}}
\nc{\BunBP}{\widetilde{\Bun}_{B,P}}
\nc{\GUb}{\overline{G/U}}
\nc{\GUPb}{\overline{G/U(P)}}

\nc{\Hhom}{\underline{\on{Hom}}}
\nc\syminfty{\on{Sym}^{\infty}}
\nc\lal{\ol{\kappa_x}}
\nc\xl{\ol{x}}
\nc\thl{\ol{\theta}}
\nc\nul{\ol{\nu}}
\nc\mul{\ol{\mu}}
\nc\Sum\Sigma
\nc{\oX}{\overset{o}{X}{}}
\nc{\hl}{\overset{\leftarrow}h{}}
\nc{\hr}{\overset{\rightarrow}h{}}
\nc{\M}{{\mathcal M}}
\nc{\N}{{\mathcal N}}
\nc{\F}{{\mathcal F}}
\nc{\D}{{\mathcal D}}
\nc{\Q}{{\mathcal Q}}
\nc{\Y}{{\mathcal Y}}
\nc{\G}{{\mathcal G}}
\nc{\E}{{\mathcal E}}
\nc{\CalC}{{\mathcal C}}
\nc\Dh{\widehat{\D}}

\nc{\C}{{\mathcal C}}
\nc{\K}{{\mathcal K}}
\renewcommand{\H}{{\mathcal H}}

\nc{\T}{{\mathcal T}}
\nc{\V}{{\mathcal V}}
\renc{\P}{{\mathcal P}}
\nc{\A}{{\mathcal A}}
\nc{\B}{{\mathcal B}}
\nc{\U}{{\mathcal U}}

\nc{\Gr}{{\on{Gr}}}

\nc{\frn}{{\check{\mathfrak u}(P)}}

\nc{\fC}{\mathfrak C}
\nc{\p}{\mathfrak p}
\nc{\q}{\mathfrak q}
\nc\f{{\mathfrak f}}

\nc{\qo}{{\mathfrak q}}
\nc{\po}{{\mathfrak p}}
\nc{\s}{{\mathfrak s}}
\nc\w{\text{w}}

\nc\mathi\iota
\nc\Spec{\on{Spec}}
\nc\Mod{\on{Mod}}
\nc{\tw}{\widetilde{\mathfrak t}}
\nc{\pw}{\widetilde{\mathfrak p}}
\nc{\qw}{\widetilde{\mathfrak q}}
\nc{\jw}{\widetilde j}

\nc{\grb}{\overline{\Gr}}
\nc{\I}{\mathcal I}

\nc{\kappach}{{\check\kappa_x}}
\nc{\Lambdach}{{\check\Lambda}{}}
\nc{\much}{{\check\mu}}
\nc{\omegach}{{\check\omega}}
\nc{\nuch}{{\check\nu}}
\nc{\etach}{{\check\eta}}
\nc{\alphach}{{\checka}}
\nc{\oblvtach}{{\check\oblvta}}
\nc{\pich}{{\check\pi}}
\nc{\ch}{{\check h}}

\nc{\Hb}{\overline{\H}}

\nc{\BA}{{\mathbb{A}}}
\nc{\BC}{{\mathbb{C}}}
\nc{\BE}{{\mathbb{E}}}
\nc{\BF}{{\mathbb{F}}}
\nc{\BG}{{\mathbb{G}}}
\nc{\BM}{{\mathbb{M}}}
\nc{\BO}{{\mathbb{O}}}
\nc{\BD}{{\mathbb{D}}}
\nc{\BN}{{\mathbb{N}}}
\nc{\BP}{{\mathbb{P}}}
\nc{\BQ}{{\mathbb{Q}}}
\nc{\BR}{{\mathbb{R}}}
\nc{\BZ}{{\mathbb{Z}}}
\nc{\BS}{{\mathbb{S}}}

\nc{\CA}{{\mathcal{A}}}
\nc{\CB}{{\mathcal{B}}}

\nc{\CE}{{\mathcal{E}}}
\nc{\CF}{{\mathcal{F}}}
\nc{\CG}{{\mathcal{G}}}
\nc{\CH}{{\mathcal{H}}}

\nc{\CL}{{\mathcal{L}}}
\nc{\CC}{{\mathcal{C}}}
\nc{\CM}{{\mathcal{M}}}
\nc{\CN}{{\mathcal{N}}}
\nc{\CK}{{\mathcal{K}}}
\nc{\CO}{{\mathcal{O}}}
\nc{\CP}{{\mathcal{P}}}
\nc{\CQ}{{\mathcal{Q}}}
\nc{\CR}{{\mathcal{R}}}
\nc{\CS}{{\mathcal{S}}}
\nc{\CT}{{\mathcal{T}}}
\nc{\CU}{{\mathcal{U}}}
\nc{\CV}{{\mathcal{V}}}
\nc{\CW}{{\mathcal{W}}}
\nc{\CX}{{\mathcal{X}}}
\nc{\CY}{{\mathcal{Y}}}
\nc{\CZ}{{\mathcal{Z}}}
\nc{\CI}{{\mathcal{I}}}
\nc{\CJ}{{\mathcal{J}}}

\nc{\csM}{{\check{\mathcal A}}{}}
\nc{\oM}{{\overset{\circ}{\mathcal M}}{}}
\nc{\obM}{{\overset{\circ}{\mathbf M}}{}}
\nc{\oCA}{{\overset{\circ}{\mathcal A}}{}}
\nc{\obA}{{\overset{\circ}{\mathbf A}}{}}
\nc{\ooM}{{\overset{\circ}{M}}{}}
\nc{\osM}{{\overset{\circ}{\mathsf M}}{}}
\nc{\vM}{{\overset{\bullet}{\mathcal M}}{}}
\nc{\nM}{{\underset{\bullet}{\mathcal M}}{}}
\nc{\oD}{{\overset{\circ}{\mathcal D}}{}}
\nc{\obD}{{\overset{\circ}{\mathbf D}}{}}
\nc{\oA}{{\overset{\circ}{\mathbb A}}{}}
\nc{\op}{{\overset{\bullet}{\mathbf p}}{}}
\nc{\cp}{{\overset{\circ}{\mathbf p}}{}}
\nc{\oU}{{\overset{\bullet}{\mathcal U}}{}}
\nc{\oZ}{{\overset{\circ}{\mathcal Z}}{}}
\nc{\ofZ}{{\overset{\circ}{\mathfrak Z}}{}}
\nc{\oF}{{\overset{\circ}{\fF}}}

\nc{\fa}{{\mathfrak{a}}}
\nc{\fb}{{\mathfrak{b}}}
\nc{\fd}{{\mathfrak{d}}}
\nc{\ff}{{\mathfrak{f}}}
\nc{\fg}{{\mathfrak{g}}}
\nc{\fgl}{{\mathfrak{gl}}}
\nc{\fh}{{\mathfrak{h}}}
\nc{\fj}{{\mathfrak{j}}}
\nc{\fl}{{\mathfrak{l}}}
\nc{\fm}{{\mathfrak{m}}}
\nc{\fn}{{\mathfrak{n}}}
\nc{\fu}{{\mathfrak{u}}}
\nc{\fp}{{\mathfrak{p}}}
\nc{\fr}{{\mathfrak{r}}}
\nc{\fs}{{\mathfrak{s}}}
\nc{\ft}{{\mathfrak{t}}}
\nc{\fz}{{\mathfrak{z}}}
\nc{\fsl}{{\mathfrak{sl}}}
\nc{\hsl}{{\widehat{\mathfrak{sl}}}}
\nc{\hgl}{{\widehat{\mathfrak{gl}}}}
\nc{\hg}{{\widehat{\mathfrak{g}}}}
\nc{\chg}{{\widehat{\mathfrak{g}}}{}^\vee}
\nc{\hn}{{\widehat{\mathfrak{n}}}}
\nc{\chn}{{\widehat{\mathfrak{n}}}{}^\vee}

\nc{\fA}{{\mathfrak{A}}}
\nc{\fB}{{\mathfrak{B}}}
\nc{\fD}{{\mathfrak{D}}}
\nc{\fE}{{\mathfrak{E}}}
\nc{\fF}{{\mathfrak{F}}}
\nc{\fG}{{\mathfrak{G}}}
\nc{\fK}{{\mathfrak{K}}}
\nc{\fL}{{\mathfrak{L}}}
\nc{\fM}{{\mathfrak{M}}}
\nc{\fN}{{\mathfrak{N}}}
\nc{\fP}{{\mathfrak{P}}}
\nc{\fU}{{\mathfrak{U}}}
\nc{\fV}{{\mathfrak{V}}}
\nc{\fZ}{{\mathfrak{Z}}}

\nc{\bb}{{\mathbf{b}}}
\nc{\bc}{{\mathbf{c}}}
\nc{\bd}{{\mathbf{d}}}
\nc{\bbf}{{\mathbf{f}}}
\nc{\be}{{\mathbf{e}}}
\nc{\bi}{{\mathbf{i}}}
\nc{\bj}{{\mathbf{j}}}
\nc{\bn}{{\mathbf{n}}}
\nc{\bp}{{\mathbf{p}}}
\nc{\bq}{{\mathbf{q}}}
\nc{\bu}{{\mathbf{u}}}
\nc{\bv}{{\mathbf{v}}}
\nc{\bx}{{\mathbf{x}}}
\nc{\bs}{{\mathbf{s}}}
\nc{\by}{{\mathbf{y}}}
\nc{\bw}{{\mathbf{w}}}
\nc{\bA}{{\mathbf{A}}}
\nc{\bK}{{\mathbf{K}}}
\nc{\bB}{{\mathbf{B}}}
\nc{\bF}{{\mathbf{F}}}
\nc{\bC}{{\mathbf{C}}}
\nc{\bG}{{\mathbf{G}}}
\nc{\bD}{{\mathbf{D}}}
\nc{\bE}{{\mathbf{E}}}
\nc{\bH}{{\mathbf{H}}}
\nc{\bO}{{\mathbf{O}}}
\nc{\bI}{{\mathbf{I}}}
\nc{\bM}{{\mathbf{M}}}
\nc{\bN}{{\mathbf{N}}}
\nc{\bV}{{\mathbf{V}}}
\nc{\bW}{{\mathbf{W}}}
\nc{\bX}{{\mathbf{X}}}
\nc{\bZ}{{\mathbf{Z}}}
\nc{\bS}{{\mathbf{S}}}

\nc{\sA}{{\mathsf{A}}}
\nc{\sB}{{\mathsf{B}}}
\nc{\sC}{{\mathsf{C}}}
\nc{\sD}{{\mathsf{D}}}
\nc{\sF}{{\mathsf{F}}}
\nc{\sK}{{\mathsf{K}}}
\nc{\sM}{{\mathsf{M}}}
\nc{\sO}{{\mathsf{O}}}
\nc{\sW}{{\mathsf{W}}}
\nc{\sQ}{{\mathsf{Q}}}
\nc{\sP}{{\mathsf{P}}}
\nc{\sZ}{{\mathsf{Z}}}
\nc{\sk}{{\mathsf{k}}}
\nc{\sr}{{\mathsf{r}}}
\nc{\bk}{{\mathsf{k}}}
\nc{\sg}{{\mathsf{g}}}
\nc{\sff}{{\mathsf{f}}}
\nc{\sfb}{{\mathsf{b}}}
\nc{\sfc}{{\mathsf{c}}}
\nc{\sd}{{\mathsf{d}}}

\nc{\BK}{{\bar{K}}}

\nc{\tA}{{\widetilde{\mathbf{A}}}}
\nc{\tB}{{\widetilde{\mathcal{B}}}}
\nc{\tg}{{\widetilde{\mathfrak{g}}}}
\nc{\tG}{{\widetilde{G}}}
\nc{\TM}{{\widetilde{\mathbb{M}}}{}}
\nc{\tO}{{\widetilde{\mathsf{O}}}{}}
\nc{\tU}{{\widetilde{\mathfrak{U}}}{}}
\nc{\TZ}{{\tilde{Z}}}
\nc{\tx}{{\tilde{x}}}
\nc{\tbv}{{\tilde{\bv}}}
\nc{\tfP}{{\widetilde{\mathfrak{P}}}{}}
\nc{\tz}{{\tilde{\zeta}}}
\nc{\tmu}{{\tilde{\mu}}}

\nc{\urho}{\underline{\pi}}
\nc{\uB}{\underline{B}}
\nc{\uC}{{\underline{\mathbb{C}}}}
\nc{\ui}{\underline{i}}
\nc{\uj}{\underline{j}}
\nc{\ofP}{{\overline{\mathfrak{P}}}}
\nc{\oB}{{\overline{\mathcal{B}}}}
\nc{\og}{{\overline{\mathfrak{g}}}}
\nc{\oI}{{\overline{I}}}

\nc{\eps}{\varepsilon}
\nc{\hrho}{{\hat{\pi}}}

\nc{\one}{{\mathbf{1}}}
\nc{\two}{{\mathbf{t}}}

\nc{\Rep}{{\mathop{\operatorname{\rm Rep}}}}
\nc{\Tot}{{\mathop{\operatorname{\rm Tot}}}}
\nc{\Ker}{{\mathop{\operatorname{\rm Ker}}}}
\nc{\Hilb}{{\mathop{\operatorname{\rm Hilb}}}}
\nc{\End}{{\mathop{\operatorname{\rm End}}}}
\nc{\Ext}{{\mathop{\operatorname{\rm Ext}}}}
\nc{\CHom}{{\mathop{\operatorname{{\mathcal{H}}\it om}}}}
\nc{\GL}{{\mathop{\operatorname{\rm GL}}}}
\nc{\gr}{{\mathop{\operatorname{\rm gr}}}}
\nc{\Id}{{\mathop{\operatorname{\rm Id}}}}
\nc{\de}{{\mathop{\operatorname{\rm def}}}}
\nc{\length}{{\mathop{\operatorname{\rm length}}}}
\nc{\supp}{{\mathop{\operatorname{\rm supp}}}}

\nc{\Cliff}{{\mathsf{Cliff}}}
\nc{\Fl}{\on{Fl}}
\nc{\Fib}{{\mathsf{Fib}}}
\nc{\Coh}{{\mathsf{Coh}}}
\nc{\QCoh}{{\on{QCoh}}}
\nc{\IndCoh}{{\on{IndCoh}}}
\nc{\FCoh}{{\mathsf{FCoh}}}

\nc{\reg}{{\text{\rm reg}}}

\nc{\cplus}{{\mathbf{C}_+}}
\nc{\cminus}{{\mathbf{C}_-}}
\nc{\cthree}{{\mathbf{C}_*}}
\nc{\Qbar}{{\bar{Q}}}
\nc\Eis{\on{Eis}}
\nc\Eisb{\ol\Eis{}}
\nc\Eisr{\on{Eis}^{rat}{}}
\nc\wh{\widehat}
\nc{\Def}{\on{Def_{\check{\fb}}(E)}}
\nc{\barZ}{\overline{Z}{}}
\nc{\barbarZ}{\overline{\barZ}{}}
\nc{\barpi}{\overline\iota}
\nc{\barbarpi}{\overline\barpi}
\nc{\barpip}{\overline\iota{}^+}
\nc{\barpim}{\overline\iota{}^-}

\nc{\fq}{\mathfrak q}

\nc{\fqb}{\ol{\fq}{}}
\nc{\fpb}{\ol{\fp}{}}
\nc{\fpr}{{\fp^{rat}}{}}
\nc{\fqr}{{\fq^{rat}}{}}

\nc{\hattimes}{\wh\otimes}

\nc{\bh}{{\bar{h}}}
\nc{\bOmega}{{\overline{\Omega(\check \fn)}}}

\nc{\seq}[1]{\stackrel{#1}{\sim}}

\nc{\cT}{{\check{T}}}
\nc{\cG}{{\check{G}}}
\nc{\cM}{{\check{M}}}
\nc{\cB}{{\check{B}}}

\nc{\ct}{{\check{\mathfrak t}}}
\nc{\cg}{{\check{\fg}}}
\nc{\cb}{{\check{\fb}}}
\nc{\cn}{{\check{\fn}}}

\nc{\cLambda}{{\check\Lambda}}

\nc{\cla}{{\check\kappa_x}}
\nc{\cmu}{{\check\mu}}
\nc{\cnu}{{\check\nu}}
\nc{\ceta}{{\check\eta}}

\nc{\DefbE}{{\on{Def}_{\cB}(E_\cT)}}

\nc{\imathb}{{\ol{\imath}}}
\nc{\rlr}{\overset{\longrightarrow}{\underset{\longrightarrow}\longleftarrow}}

\nc{\KG}{K\backslash G}
\nc{\comult}{{co\text{-}mult}}
\nc{\counit}{{co\text{-}unit}}
\nc{\uHom}{{\underline{\Maps}}}
\nc{\dgSch}{\on{Sch}}
\nc{\Sch}{\on{Sch}}
\nc{\affdgSch}{\on{Sch}^{\on{aff}}}
\nc{\affSch}{\on{Sch}^{\on{aff}}}
\nc{\Groupoids}{\on{Grpd}}
\nc{\inftygroup}{\on{Spc}}
\nc{\inftyCat}{\infty\on{-Cat}}
\nc{\StinftyCat}{\inftyCat^{\on{St}}}
\nc{\MoninftyCat}{\infty\on{-Cat}^{\on{Mon}}}
\nc{\SymMoninftyCat}{\infty\on{-Cat}^{\on{SymMon}}}
\nc{\SymMonStinftyCat}{\on{DGCat}^{\on{SymMon}}}
\nc{\MonStinftyCat}{\on{DGCat}^{\on{Mon}}}
\nc{\inftystack}{\on{Stk}}
\nc{\inftystackalg}{Stk^{1\text{-}alg}}
\nc{\inftyprestack}{\on{PreStk}}
\nc{\inftydgnearstack}{\on{NearStk}}
\nc{\inftydgstack}{\on{Stk}}
\nc{\inftydgstackalg}{DGStk^{1\text{-}alg}}
\nc{\inftydgprestack}{\on{PreStk}}
\nc{\dgindSch}{\on{indSch}}
\nc{\indSch}{{}^{\on{cl}}\!\on{indSch}}
\nc{\infSch}{\on{infSch}}
\nc{\dr}{{\on{dR}}}

\nc{\mmod}{{\on{-}\!{\mathbf{mod}}}}

\nc{\starr}{\text{\dh}}
\nc{\Spectra}{\on{Spectra}}
\nc{\Crys}{\on{Crys}}
\nc{\oblv}{{\mathbf{oblv}}}
\nc{\ind}{{\mathbf{ind}}}
\nc{\CMaps}{{\mathcal Maps}}
\nc{\Maps}{\on{Maps}}
\nc{\bMaps}{\mathbf{Maps}}
\nc{\BMaps}{\ul{\on{Maps}}}
\nc{\Grid}{\on{Grid}}
\nc{\hGrid}{\on{Grid}^{\geq\,\on{dgnl}}}
\nc{\Diag}{\on{Diag}}
\nc{\bDelta}{\mathbf{\Delta}}
\nc{\tCateg}{(\infty\on{-2)-Cat}}
\nc{\ul}{\underline}
\nc{\Seg}{\on{Seq}}
\nc{\biSeg}{\on{bi-Seq}}
\nc{\triSeg}{\on{tri-Seq}}
\nc{\quadSeg}{\on{quad-Seq}}
\nc{\nSeg}{\on{n-Seq}}
\nc{\Segm}{\on{Seg}^{\on{mkd}}}
\nc{\fLm}{\fL^{\on{mkd}}}
\nc{\inftyCatm}{\inftyCat^{\on{mkd}}}
\nc{\Blocks}{\mathbf{Blocks}}
\nc{\Snakes}{\mathbf{Snakes}}
\nc{\bifL}{\on{bi-}\!\fL}
\nc{\Sets}{\on{Sets}}
\nc{\Ran}{\on{Ran}}
\nc{\Vect}{\on{Vect}}
\nc{\Shv}{\on{Shv}}
\nc{\unn}{\mathbf{union}}
\nc{\Spc}{\on{Spc}}
\nc{\ppart}{(\!(t)\!)}
\nc{\qqart}{[\![t]\!]}
\nc{\Dmod}{\on{D-mod}}
\nc{\cD}{\mathcal D}
\nc{\ocD}{\overset{\circ}{\cD}}
\nc{\sfp}{\mathsf{p}}
\nc{\sfq}{\mathsf{q}}
\nc{\DGCat}{\on{DGCat}}
\renc{\det}{\on{det}}

\begin{document}

\title[Parameterization of factorizable line bundles]
{Parameterization of factorizable line bundles \\ by $K$-theory and motivic cohomology}

%\title{How to solve all problems once and for all} 

\author{Dennis Gaitsgory}

\dedicatory{To Sasha Beilinson}

\date{\today}

%\begin{abstract}
%This note corrects a certain construction in (the old version of) the paper \cite{GL}. The construction in question starts from a 
%\emph{Brylinski-Deligne datum}, which is an extension of a group $G$ be $K_2$, and produces a factorization line bundle
%on the affine Grassmannian of the group $G$. 
%\end{abstract} 

\maketitle

\tableofcontents

\section*{Introduction}

\ssec{The context}

This paper grew out of \cite{GL}, where the main construction that we will perform here was stated erroneously.
Let us explain what the paper \cite{GL} aimed to do, and what the objective of the present one is.

\sssec{}

In \cite{GL} our objective was to set up the context for the \emph{metaplectic geometric Langlands theory}.

\medskip

Namely, given a connected reductive group $G$ and an algebraic curve $X$ (say, over a field $k$), 
it was proposed that a \emph{geometric metaplectic datum} for the pair $(G,X)$ should be understood as a \emph{factorizable gerbe}
on the affine Grassmannian $\Gr_G$ of $G$ over $X$. 

\medskip

Let us recall that, whatever the geometric Langlands theory is, it has to do with sheaves on various spaces 
(here a ``space" can mean a scheme, ind-scheme, stack or a general prestack) attached to
the pair $(G,X)$. Thus, when we say ``gerbe" we mean a gerbe banded by some group, which has to do with 
coefficients of our sheaves. 

\medskip

Let us say that our sheaf theory is that of $\ell$-adic sheaves. Then the group in question is
$(\ol\BQ_\ell^\times)_{\on{tors}}$, the group of roots of unity in $\ol\BQ_\ell$ of order prime to $\on{char}(k)$. 
Having a gerbe $\CG$ banded by $(\ol\BQ_\ell^\times)_{\on{tors}}$ on a prestack $\CY$, we can consider the
category $\on{Shv}_\ell(\CY)_\CG$, the stable $\infty$-category of $\ell$-adic sheaves on $\CY$, twisted by $\CG$. 

\medskip

In the metaplectic geometric Langlands theory we wish to have a procedure that attaches $(\ol\BQ_\ell^\times)_{\on{tors}}$-gerbes
to all the prestacks naturally associated with $(G,X)$ (such as $\Bun_G(X)$, the moduli stack of principal  $G$-bundles on $X$
if the latter is complete), in a consistent way. And it turns out that a factorizable gerbe on $\Gr_G$ is precisely a datum that gives rise 
to such a procedure. 

\medskip

Moreover, if $k$ is a finite field $\BF_q$, a geometric metaplectic datum gives rise to a metaplectic extension of $G(F_x)$,
where $x$ is a closed point of $X$, and $F_x$ is the corresponding local field. 

\sssec{}

Factorizable gerbes on the affine Grassmannian $\Gr_G$ form what is called a \emph{Picard 2-category}: a symmetric monoidal
2-category in which every object is invertible. %We denote it by $\on{FactGerbe}_{(\ol\BQ_\ell^\times)_{\on{tors}}}(\Gr_G)$. 
Now, one of the key observations is that this Picard 2-category is relatively easy to describe explicitly (see \cite[Sect. 3]{GL}): 

\medskip

Namely, it is the category of \'etale 4-cocyles on $X\times BG$ (relative to the base point of $BG$) 
with coefficients in $(\ol\BQ_\ell^\times)_{\on{tors}}(1)$  
$$\tau^{\leq 4}\biggl(\Gamma_{\on{et}}(X\times B_{\on{et}}(G);X\times \on{pt},(\ol\BQ_\ell^\times)_{\on{tors}}(1))\biggr),$$
where $(1)$ is the Tate twist. In the above formula $B_{\on{et}}(G)$ is the usual algebraic
stack version of the classifying space, i.e., the sheafification of the prestack quotient $\on{pt}/G$ in the \'etale topology;
but we can replace it by either $B_{\on{Zar}}(G)$ (=sheafification of $\on{pt}/G$ in the Zariski topology) or 
$\on{pt}/G$ itself without changing the answer. 

\medskip 

In particular, the set of equivalence classes of such gerbes is 
$$H^4_{\on{et}}(X\times B_{\on{et}}(G);X\times \on{pt},(\ol\BQ_\ell^\times)_{\on{tors}}(1)),$$
etc. 

\sssec{}

Here is, however, where the current paper grew out from.
Other authors (see, e.g., \cite{We}) used a different object to parameterize metaplectic extensions. 

\medskip

Namely, they start with what we call
a \emph{Brylinski-Deligne datum}, which means by definition an extension of the sheaf of groups $G\times X$ over $X$ by $\CK_2$,
where the latter is the Zariski sheafification of the presheaf that attaches $S\in \on{Sch}_{/X}$ the group $K_2(S)$. 

\medskip

Brylinski-Deligne data form a \emph{Picard category} (i.e., a symmetric monoidal category in which every object is invertible).
We can tautologically rewrite it as the space of maps between prestacks 
\begin{equation} \label{e:BrDe}
\Maps_{\on{based}}(X\times B_{\on{Zar}}(G),B^2_{\on{Zar}}(\bK_2)),
\end{equation}
where $B^2_{\on{Zar}}(\bK_2)$
is the Zariski sheafification of the double classifying space of the commutative group prestack $\bK_2$ that attaches to $S$ the group $\CK_2(S)$. 
In the above formula we could have replaced $B_{\on{Zar}}(G)$ by prestack quotient $\on{pt}/G$ without changing the answer. The
subscript ``based" means that we are considering the space of maps that are required to send $X\times \on{pt}$ to the base-point of 
$B^2_{\on{Zar}}(\bK_2)$. 

\medskip

The paper \cite{BrDe} gives an explicit description of the Picard category \eqref{e:BrDe},
and, in particular, of its homotopy groups. We recall this description in \secref{ss:BrDe}.

\begin{rem} \label{r:Zar et}
We should emphasize that the Picard category $\Maps_{\on{based}}(B_{\on{Zar}}(G),B^2_{\on{Zar}}(\bK_2))$ is quite sensitive to 
the topologies we consider. 

\medskip

For example, if we replace replace $B_{\on{Zar}}(G)$ by $B_{\on{et}}(G)$, we will have a natural map 
$$\Maps_{\on{based}}(B_{\on{et}}(G),B^2_{\on{Zar}}(\bK_2))\to \Maps_{\on{based}}(B_{\on{Zar}}(G),B^2_{\on{Zar}}(\bK_2)),$$
which at the level of $\pi_0$ is an injection with finite cokernel, see Remark \ref{r:final}. This is essentially taken from \cite{To}. 

\medskip

If we replace $B^2_{\on{Zar}}(\bK_2)$ by $B^2_{\on{et}}(\bK_2)$, then we will have the (tautological) identification
%\begin{multline*}
$$\Maps_{\on{based}}(B_{\on{et}}(G),B^2_{\on{et}}(\bK_2))\simeq \Maps_{\on{based}}(B_{\on{Zar}}(G),B^2_{\on{et}}(\bK_2)),$$
while the latter is isomorphic (at least after inverting $\on{char}(k)$) with its rationalization
$$\Maps_{\on{based}}(B_{\on{Zar}}(G),B^2_{\on{et}}(\bK_2))\otimes \BQ$$
(indeed, it follows from Suslin's rigidity that $B^2_{\on{et}}(\bK_2)$ is uniquely $n$-divisible for any integer $n$ coprime with $\on{char}(k)$), 
and the latter identifies also with 
$$\Maps_{\on{based}}(B_{\on{Zar}}(G),B^2_{\on{Zar}}(\bK_2))\otimes \BQ,$$
%\end{multline*}
due to the fact that $\CK_2\otimes \BQ$ satisfies \'etale descent. 

\medskip

That said, we shall see that if we replace $\CK_2$ by motivic cohomology, the answer is less sensitive to the choice of topology,
see Remark \ref{r:mot comp} below. This is essentially taken from \cite{EKLV}.

\end{rem}

\sssec{}

The natural question is, therefore, what is the relation between the Picard 2-category
$$\on{FactGerbe}_{(\ol\BQ_\ell^\times)_{\on{tors}}}(\Gr_G) \simeq 
\tau^{\leq 4}\biggl(\Gamma_{\on{et}}(X\times B_{\on{Zar}}(G);X\times \on{pt},(\ol\BQ_\ell^\times)_{\on{tors}}(1))\biggr)$$
and the Picard category 
$$\Maps_{\on{based}}(X\times B_{\on{Zar}}(G),B^2_{\on{Zar}}(\bK_2)).$$

Ideally, we would like to see that the construction of metaplectic covers strarting from a Brylinski-Deligne datum
passes via a geometric metaplectic datum. And this is what the present paper is about. 

\sssec{}

We introduce the notion of \emph{factorizable line bundle} on $\Gr_G$. The totality of such form a Picard 
category, denoted $\on{FactLine}(\Gr_G)$. 

\medskip

The main focus of this paper is the construction of a functor of Picard categories
\begin{equation} \label{e:main}
\Maps_{\on{based}}(X\times B_{\on{Zar}}(G),B^2_{\on{Zar}}(\bK_2))\to \on{FactLine}(\Gr_G).
\end{equation} 

It is this functor, whose construction in \cite{GL} had an error. Moreover, in \conjref{c:main} we propose
that \eqref{e:main} is actually an equivalence. 

\sssec{}

Here is how the functor \eqref{e:main} allows to relate Brylinski-Deligne data and factorizable gerbes. Fix an integer $n$ 
invertible in $k$. The Kummer sequence
$$0\to \mu(n)\to \CO^\times \to \CO^\times\to 0$$
defines a functor map 
$$\on{FactLine}(\Gr_G)\to \on{FactGerbe}_{\mu(n)}(\Gr_G).$$

\medskip

Furthermore, in \secref{ss:mot} we construct a map
\begin{equation}  \label{e:K to mu intr}
\Maps_{\on{based}}(X\times B_{\on{Zar}}(G),B^2_{\on{Zar}}(\bK_2)) \to 
\tau^{\leq 4}\biggl(\Gamma_{\on{et}}(X\times B_{\on{Zar}}(G);X\times \on{pt},\mu(n)(1))\biggr),
\end{equation}
where we note that $\mu(n)(1)\simeq \mu(n)^{\otimes 2}$. 

\medskip

We propose a conjecture (which does not seem far-fetched at all) to the effect that the diagram
$$
\CD
\Maps_{\on{based}}(X\times B_{\on{Zar}}(G),B^2_{\on{Zar}}(\bK_2)) @>>> 
\tau^{\leq 4}\biggl(\Gamma_{\on{et}}(X\times B_{\on{Zar}}(G);X\times \on{pt},\mu(n)(1)\biggr)  \\
@VVV   @VVV  \\
\on{FactLine}(\Gr_G)  @>>>  \on{FactGerbe}_{\mu(n)}(\Gr_G) 
\endCD
$$
commutes. 

\sssec{}  \label{sss:integer N}

That said, our construction of the functor \eqref{e:main} has a small caveat. Namely, we need impose the condition that $\on{char}(k)$
is comprime with a certain integer, denoted $N$. Here is what this integer is: 

\medskip

Let $G_\BC$ be the complex Lie group corresponding to $G$.
We consider $G_\BC$ as a topological group, and consider its (topological) classifying space $BG_{\on{top}}$. Consider the abelian group
$$H^4(BG_{\on{top}},\BZ).$$
Inside we can consider the subgroup spanned by Chern classes of representations, i.e., the span of the images 
of the maps
$$H^4(BGL(V)_{\on{top}},\BZ)\to H^4(BG_{\on{top}},\BZ)$$
for all finite-dimensional representations $V$ of $G$.

\medskip

When tensored with $\BQ$, this subgroup becomes all of $H^4(BG_{\on{top}},\BZ)$. Hence, it has a finite index.
We let $N$ be smallest integer that annihilates the quotient.

\medskip

This integer is tautologically equal to $1$ for $G=GL_n$. One can show that it is also equal to $1$ for $G=Sp(2n)$.
But it is \emph{not} equal to $1$ for semi-simple simply-connected groups of other types. 

\ssec{Structure of the paper}

\sssec{}

As was explained above, the main goal of this paper is to construct the map \eqref{e:main}. This is supposed to be very simple:

\medskip

By unwinding the definitions, what we need to do is the following. Let $S$ be an affine scheme and let $I$ be a finite set
of maps $S\to X$. Let $U_I\subset S\times X$ be the complement to the union of their graphs. Suppose we are given a
2-cocycle on $S\times X$ with coefficients in $\CK_2$, equipped with a trivialization on $U_I$. To this datum wish to
attach its \emph{residue} or \emph{trace}, which would be a 1-cocyle on $S$ with coefficients in $\CK_1\simeq \CO_S^\times$,
i.e., a line bundle on $S$.

\medskip

The problem is, however, that we were not able to construct such a residue map:
\begin{equation} \label{e:res}
\tau^{\leq 0}\biggl(\Gamma(S\times X;U_I,\CK_2[2])\biggr) \to \Gamma(S,\CK_1)[1].
\end{equation} 

\medskip

Our inability to do so is due to the lack of control of K-theory on schemes that are non-regular (we need to allow $S$ non-regular, because it
will typically be a test-scheme mapping to $\Gr_G$). This is where the mistake in \cite{GL} was. 

\medskip

What we can do is to construct some particular cases of \eqref{e:res}, specifically, when $S$ is smooth or a fat point 
(=spectrum of a local Artinian ring). This is done in \secref{s:particular}. 

\medskip

The combination of these two cases will ``almost" give us the desired map \eqref{e:main},
but not quite. 

\sssec{}

We will need another variant of \eqref{e:res}, where instead of $\CK_2$ we use the 2-connective truncation 
of the \emph{K-theory spectrum}:
\begin{equation} \label{e:res full}
\Gamma(S\times X;U_I,\tau^{\leq -2}(\CK))\to \Gamma(S,\CK_1)[1].
\end{equation}  

\medskip

The reason we could construct the map \eqref{e:res full} (as opposed to \eqref{e:res}) is that the full K-theory
spectrum has better functoriality properties than the individual sheaves $\CK_i$, see \secref{s:spectrum}. 

\medskip

The map \eqref{e:res full} still does not give rise (at least, not in a way that we could see) to a map \eqref{e:res},
because we are not able to prove that it vanishes on $\Gamma(S\times X;U_I,\tau^{\leq -3}(\CK))$.  

\medskip

However, playing the map \eqref{e:res full} against the cases in which we could construct the map
\eqref{e:res}, we will able to carry out the construction of the map \eqref{e:main}. This is done in
Sects. \ref{s:construction abs} and \ref{s:construction gen}.

\sssec{} 

In order to use the map \eqref{e:res full}, we will have to lift a Brylinski-Deligne datum, which is a based map
$$\kappa:X\times B_{\on{Zar}}(G)\to B^2_{\on{Zar}}(\bK_2),$$ to a map
$$\kappa':X\times B_{\on{Zar}}(G)\to \bK_{\geq 2}.$$

Here $\bK_{\geq 2}$ is the prestack defined by 
$$\Maps(S,\bK_{\geq 2}):=\Gamma(S,\tau^{\leq -2}(\CK)), \quad S\in \Sch^{\on{aff}}.$$

We could not prove that such a lift always exists, and this is where the integer $N$ of \secref{sss:integer N}
enters the game: 

\medskip

We show that $N\cdot \kappa$ can always be lifted, see \thmref{t:raising to power}. The
fact that it is sufficient to consider $N\cdot \kappa$ instead of $\kappa$ itself is the ``raising to the power"
trick, suggested to us by A.~Beilinson. 

\sssec{}

Finally, in \secref{s:mot}, we construct the map \eqref{e:K to mu intr}. The construction is indirect, and it passes 
by comparing the Picard category $\Maps_{\on{based}}(X\times B_{\on{Zar}}(G),B^2_{\on{Zar}}(\bK_2))$ with its
motivic cohomology counterpart (this idea was also suggested to us by A.~Beilinson). 

\medskip

Namely, we prove (reproducing an argument from \cite{EKLV}) that the natural map
$$\Gamma(X\times B_{\on{Zar}}(G),\BZ_{\on{mot}}(2)[2])\to  
\CK_2(X\times B_{\on{Zar}}(G))$$
induces an isomorphism on the $\tau^{\geq 0}$ truncations. We then use the composite
\begin{multline*}
\Gamma(X\times B_{\on{Zar}}(G);X\times \on{pt},\BZ_{\on{mot}}(2)[4])\to \Gamma_{\on{et}}(X\times B_{\on{Zar}}(G);X\times \on{pt},\BZ_{\on{mot}}(2)[4])\to \\
\to \Gamma_{\on{et}}(X\times B_{\on{Zar}}(G);X\times \on{pt},\BZ/n(2)[4])\simeq 
\Gamma_{\on{et}}(X\times B_{\on{Zar}}(G);X\times \on{pt},\mu_n^{\otimes 2}[4])
\end{multline*}
to obtain the map \eqref{e:K to mu intr}.

\medskip

We also reproduce the proof of another result of \cite{EKLV}, which says that the map
$$\Gamma(X\times B_{\on{Zar}}(G);X\times \on{pt},\BZ_{\on{mot}}(2))\to \Gamma_{\on{et}}(X\times B_{\on{Zar}}(G);X\times \on{pt},\BZ_{\on{mot}}(2))$$
induces an isomorphism on the $\tau^{\leq 4}$ truncations. 

\begin{rem}  \label{r:mot comp}
To summarize, we obtain that the equivalent Picard 2-categories
\begin{multline*} 
\tau^{\geq -2,\leq 0}\biggl(\Gamma(X\times B_{\on{Zar}}(G);X\times \on{pt},\BZ_{\on{mot}}(2)[4])\biggr) \simeq 
\tau^{\geq -2,\leq 0}\biggl(\Gamma_{\on{et}}(X\times B_{\on{Zar}}(G);X\times \on{pt},\BZ_{\on{mot}}(2)[4])\biggr) \simeq \\
\simeq \tau^{\geq -2,\leq 0}\biggl(\Gamma_{\on{et}}(X\times B_{\on{et}}(G);X\times \on{pt},\BZ_{\on{mot}}(2)[4])\biggr)
\simeq \tau^{\geq -2,\leq 0}\biggl(\Gamma_{\on{et}}(X\times (BG)^{\on{appr}};X\times \on{pt},\BZ_{\on{mot}}(2)[4])\biggr)
\end{multline*}
(here $(BG)^{\on{appr}}$ is the ``scheme approximation" of $B_{\on{et}}(G)$ from \secref{sss:E and V}), which are also 
equivalent to our 
$$\Maps_{\on{based}}(X\times B_{\on{Zar}}(G),B^2_{\on{Zar}}(\bK_2)).$$

\medskip

By contrast, the Picard 2-categories
$$\tau^{\geq -2,\leq 0}\biggl(\Gamma(X\times B_{\on{et}}(G);X\times \on{pt},\BZ_{\on{mot}}(2)[4])\biggr) \simeq 
\tau^{\geq -2,\leq 0}\biggl(\Gamma(X\times (BG)^{\on{appr}};X\times \on{pt},\BZ_{\on{mot}}(2)[4])\biggr)$$
and 
$$\Maps_{\on{based}}(X\times B_{\on{et}}(G),B^2_{\on{Zar}}(\bK_2)) \simeq \Maps_{\on{based}}(X\times (BG)^{\on{appr}},B^2_{\on{Zar}}(\bK_2)),$$
which are equivalent to \emph{each other}, but they are, in general, \emph{different from the one above}, see Remarks \ref{r:Zar et} and 
\ref{r:final}. 
\end{rem}

\ssec{Acknowledgements}

Special thanks are due to Sasha Beilinson, who suggested to me the key ideas to overcome the obstacles in 
carrying out the construction in this paper (to use the full K-theory spectrum, the ``raising to the power" trick, 
and the comparison with motivic cohomology). 
No less importantly, I would like to thank him for his patience over the years
in answering my persistent questions, along with teaching me the basics of motivic cohomology. 

\medskip

I am also very grateful to S.~Bloch, D.~Clausen, E.~Elmanto, H.~Esnault, M.~Hopkins, M.~Kerr, J.~Lurie, A.~Mathew and 
C.~Weibel for patiently answering my questions on K-theory and motivic cohomology. 

\medskip

Finally, I would like to thank my graduate student Y.~Zhao for catching the original mistake in \cite{GL}, as well as for
numerous subsequent discussions on the subject.  I am additionally grateful to him and J.~Tao for taking up (what
appears in the present paper as) Conjecture 6.1.2  as a project. 

\medskip

The author is supported by NSF grant DMS-1063470. He has also received support from ERC grant 669655.

\section{Some background, notation and conventions}

\ssec{The algebro-geometric setting}

\sssec{}

We will be working over a ground field $k$, assumed algebraically closed. We denote 
$$\on{pt}:=\Spec(k).$$

\medskip

We will denote by $\Sch^{\on{aff}}\subset \Sch$ the category of schemes (resp., affine schemes) over $k$.
We denote $\on{pt}:=\Spec(k)$. 

\medskip

\sssec{Higher category theory}

In this paper we freely use the language of higher category theory and higher algebra as developed in \cite{Lu1,Lu2}. 

\medskip

We denote by $\on{Spc}$ the $\infty$-category of spaces. Its objects can be thought of as higher groupoids. 

\medskip

For an $\infty$-category $\bC$ and objects $\bc_1,\bc_2\in \bC$, we let $\Maps(\bc_1,\bc_2)\in \on{Spc}$ denote the
space of maps between them. If $\bC$ is classical (i.e., an ordinary category), then $\Maps(\bc_1,\bc_2)$ are \emph{sets}
(=discrete spaces), and we will also write $\Hom(\bc_1,\bc_2)$. 

\medskip

Given a stable $\infty$-category $\bC$ with a t-structure, we will denote by $\tau^{\geq k},\tau^{\leq k}$, etc., the corresponding
truncation functors. We will denote by 
$$H^k(-):\bC\to \bC^\heartsuit$$ 
the functor of $k$-th cohomology with respect to this t-structure. 

\medskip

Let $\on{Sptr}$ denote the stable $\infty$-category of spectra. It comes equipped with a natural t-structure. 
For a spectrum $\CS$, we will also use the notation $\CS_{\geq k}:=\tau^{\leq -k}(\CS)$ (this is the $k$-connective truncation of $\CS$),
and $\CS_{\leq -k}:=\tau^{\geq k}(\CS)$ (this is the $k$-th Postnikov truncation of $\CS$). 

\sssec{Prestacks}

By \emph{prestack} we shall mean an arbitrary (accessible) functor 
$$(\Sch^{\on{aff}})^{\on{op}}\to \on{Spc}.$$

The $\infty$-category of prestacks is denoted $\on{PreStk}$. 

\medskip

All other algebro-geometric objects that we will encounter (schemes, ind-schemes, algebraic stacks)
are particular cases of prestacks (which means that they are determined by their Grothendieck functors
of points evaluated on affine schemes). 

\medskip

Given prestack $\CY_1$ and $\CY_2$, following our conventions, we will denote by
$$\Maps(\CY_1,\CY_2)\in \Spc$$
the space of maps from $\CY_1$ to $\CY_2$. 

\sssec{}

The assignement 
$$S\mapsto \QCoh(S), \quad (S_1\overset{f}\to S_2) \mapsto \QCoh(S_2) \overset{f^*}\to \QCoh(S_1)$$
is a functor 
$$(\Sch^{\on{aff}})^{\on{op}}\to \infty\on{-Cat}.$$

Applying the procedure of \emph{right Kan extension} along
$$(\Sch^{\on{aff}})^{\on{op}}\hookrightarrow (\on{PreStk})^{\on{op}}$$
we obtain a functor
$$\QCoh:(\on{PreStk})^{\on{op}}\to \infty\on{-Cat}.$$

In particular, for $\CY\in \on{PreStk}$ we have a well-defined category $\QCoh(\CY)$, and for a 
map $f:\CY_1\to \CY_2$ we have the pullback functor $f^*:\QCoh(\CY_2)\to \QCoh(\CY_1)$.

\medskip

Thus, one can talk about vector (resp., line) bundles on a prestack, etc: an object $\CF\in \QCoh(\CY)$
is a vector (resp., line) bundle if for any $S\to \Sch^{\on{aff}}|_{/\CY}$, the resulting object $\CF_S\in \QCoh(S)$
is a vector (resp., line) bundle.  

\sssec{}

Let $G$ be a connected reductive group over $k$.

\medskip

We will consider several different versions of the classifying space of $G$. We let $\on{pt}/G$ denote the
prestack quotient, i.e.,
$$\Maps(S,\on{pt}/G):=B(\Maps(S,G)).$$

We let $B_{\on{Zar}}(G)$ the Zariski sheafification of $\on{pt}/G$. 

\medskip

We let $B_{\on{et}}(G)$ denote the \'etale sheafification of $\on{pt}/G$; this is the usual \emph{algebraic stack}
version of the classifying space of $G$. 

\sssec{}

We let $X$ be a smooth connected algebraic curve over $k$. If $X$ is complete, we let $\Bun_G(X)$ denote the
stack of principal $G$-bundles on $X$:
$$\Maps(S,\Bun_G(X)):=\Maps(S\times X,B_{\on{et}}(G)).$$

\sssec{The Ran space}

The prestack of particular importance for us is the Ran space (say, of an algebraic curve $X$), denoted $\Ran(X)$.  By definition
$\Maps(S,\Ran(X))$ is the set (=discrete space) whose elements are finite non-empty subsets in $\Hom(S,X)$. 

\medskip

We can write $\Ran(X)$
explicitly as a colimit of schemes:
\begin{equation} \label{e:Ran}
\Ran(X)\simeq \underset{I}{\on{colim}}\, X^I,
\end{equation} 
where the colimit is taken over the category (opposite to that) of finite non-empty sets and surjective maps.

\medskip

The operation of union of finite sets defines a map 
$$\Ran(X)\times \Ran(X)\to \Ran(X).$$

\medskip

In what follows we will also need the open substack
$$(\Ran(X)\times \Ran(X))_{\on{disj}}\subset \Ran(X)\times \Ran(X)$$
defined as follows: 

\medskip

And $S$ point $(I_1,I_2)\in \Maps(S,\Ran(X)\times \Ran(X))=\Maps(S,\Ran(X))\times \Maps(S,\Ran(X))$ belongs to 
$(\Ran(X)\times \Ran(X))_{\on{disj}}$ if for any $i_1\in I_1$ and $i_2\in I_2$, the corresponding two maps
$$S\rightrightarrows X$$
have non-intersecting graphs (equivalently, for any $k$-point of $S$, the resulting two $k$-points of $X$ are distinct). 

\ssec{The affine Grassmannian}

\sssec{}

Let $I$ be a finite set. We define the affine Grassmanian $\Gr_{G,I}$ to be the ind-scheme that classifies the
data of triples
$$\Maps(S,\Gr_{G,I}):=(x_I,\CP_G,\alpha),$$
where

\begin{itemize}

\item $x_I\in \Maps(S,X)^I$;

\item $\CP_G$ is a principle $G$-bundle on $S\times X$;

\item $\alpha$ is a trivialization of
$\CP_G$ over the complement to the union of the graphs of the maps that comprise $x_I$. 

\end{itemize} 

\begin{rem} \label{r:Gr class}

Note that there is a potential ambiguity in the above definition. Namely, we could perceive $\Gr_{G,I}$ as
a functor on classical affine schemes or derived affine schemes. A priori, we obtain two \emph{different}
algebro-geometric objects, and the notions of line bundles are different (undoubtedly, the derived one is
the a priori better notion). 

\medskip

However, the above two versions are actually isomorphic (i.e., the derived version is \emph{classical} as
a prestack, up to Zariski sheafification). 
This is proved in \cite[Theorem 9.3.4]{GR}. This essentially follows from the fact that $\Gr_{G,I}$ is formally smooth.

\end{rem} 

\sssec{}

Let $I=I_1\sqcup I_2$, so that $X^I\simeq X^{I_1}\times X^{I_2}$. Let 
$$(X^{I_1}\times X^{I_2})_{\on{disj}}\subset X^{I_1}\times X^{I_2}$$
be the open subset corresponding to the condition that for any $i_1\in I_1$ and $i_2\in I_2$, the corresponding two maps
$$S\rightrightarrows X$$
have non-intersecting graphs.

\medskip

A basic feature of the affine Grassmannian is it \emph{factorization structure}, i.e., a system of isomorphisms:
\begin{equation} \label{e:fact Gr}
(\Gr_{G,I_1}\times \Gr_{G,I_2})\underset{X^{I_1}\times X^{I_2}}\times 
(X^{I_1}\times X^{I_2})_{\on{disj}} \simeq \Gr_{G,I}\underset{X^I}\times (X^{I_1}\times X^{I_2})_{\on{disj}}.
\end{equation} 

These isomorphisms are associative with respect to further partitions, in the natural sense. 

\sssec{}

We introduce the Ran version of the affine Grassmannian $\Gr_{G,\Ran}$ as follows:
$$\Maps(S,\Gr_{G,\Ran}):=(I,\CP_G,\alpha),$$
where 
\begin{itemize}

\item $I$ is a finite non-empty set of maps $S\to X$; 

\item $\CP_G$ is a principle $G$-bundle on $S\times X$;

\item $\alpha$ is a trivialization of
$\CP_G$ over the complement to the union of the graphs of the maps that comprise $I$. 

\end{itemize} 

We have the tautological map $\Gr_{G,\Ran}\to \Ran(X)$. 

\medskip

We can write 
$$\Gr_{G,\Ran}\simeq \underset{I}{\on{colim}}\, \Gr_{G,I},$$
where the colimit is taken over the same index category as for $\Ran(X)$ itself, see \eqref{e:Ran}. 

\medskip

The factorization structure on $\Gr_{R,\Ran}$ is encoded by the isomorphism 
\begin{equation} \label{e:fact Gr Ran}
(\Gr_{G,\Ran}\times \Gr_{G,\Ran})\underset{\Ran\times \Ran}\times (\Ran(X)\times \Ran(X))_{\on{disj}}\simeq
\Gr_{G,\Ran}\underset{\Ran}\times (\Ran(X)\times \Ran(X))_{\on{disj}},
\end{equation} 
which is associative in the natural sense. 

\sssec{}

The main object of study in this paper are \emph{factorizable line bundles} on $\Gr_{G,\Ran}$. By definition,
these are line bundles $\CL_{\Gr}$ on $\Gr_{G,\Ran}$, equipped with an isomorphism
$$\CL_{\Gr}\boxtimes \CL_{\Gr}|_{(\Gr_{G,\Ran}\times \Gr_{G,\Ran})\underset{\Ran\times \Ran}\times (\Ran(X)\times \Ran(X))_{\on{disj}}}
\simeq \CL_{\Gr}|_{\Gr_{G,\Ran}\underset{\Ran}\times (\Ran(X)\times \Ran(X))_{\on{disj}}}$$
(where we identify the two prestacks over which the isomorphism takes place via \eqref{e:fact Gr Ran}), and such that
the natural compatibility condition holds. 

\medskip

Equivalently, a factorizable line bundle $\CL_{\Gr}$ is an assignment for every finite non-empty set $I$ of a line bundle
$\CL^I_{\Gr}$ on $\Gr_{G,I}$, equipped with (a transitive system of) identifications
\begin{equation} \label{e:line bundle Gr}
\CL^I_{\Gr}|_{\Gr_{G,J}}\simeq \CL^J_{\Gr} \text{ for every surjection } I\twoheadrightarrow J
\end{equation} 
(this defines a line bundle on $\Gr_{G,\Ran}$), and the factorization isomorphisms
\begin{equation} \label{e:fact line bundle}
\CL^{I_1}_{\Gr}\boxtimes \CL^{I_2}_{\Gr}|_{(\Gr_{G,I_1}\times \Gr_{G,I_2})\underset{X^{I_1}\times X^{I_2}}\times 
(X^{I_1}\times X^{I_2})_{\on{disj}}}\simeq \CL^I_{\Gr}|_{\Gr_{G,I}\underset{X^I}\times (X^{I_1}\times X^{I_2})_{\on{disj}}},
\end{equation} 
(where we identify the two prestacks over which the isomorphism takes place via \eqref{e:fact Gr}), such that the identifications
\eqref{e:fact line bundle} are associative and compatible with \eqref{e:line bundle Gr} in the natural sense. 

\medskip

We denote the Picard groupoid of factorizable line bundles on $\Gr_{G,\Ran}$ by 
$$\on{FactLine}(\Gr_{G,\Ran}).$$

\ssec{Presheaves and sheaves}  \label{ss:sheaves}

\sssec{}

For a scheme $Y$ we will consider the stable $\infty$-category $\on{PreShv}(Y,\on{Ab})$ of presheaves
of chain complexes abelian groups on $Y$.

\medskip

Let the symbol ``?" denote either Zariski or \'etale topology on $Y$. Then $\on{PreShv}(Y,\on{Ab})$
contains a full subcategory 
$$\on{Shv}_?(Y,\on{Ab})\subset \on{PreShv}(Y,\on{Ab})$$
of objects that satisfy ?-descent. 

\medskip

This inclusion has a left adjoint, the sheafification functor. We will denote it by $\CF\mapsto \CF_?$. By a slight abuse
of notation, for $\CF\in \on{PreShv}(Y,\on{Ab})$ and an object $U$ of the ?-site of $Y$, we will sometimes write 
$\Gamma_?(U,\CF)$ for $\Gamma(U,\CF_?)$. 

\sssec{}

The category $\on{PreShv}(Y,\on{Ab})$ has a natural t-structure, whose heart $\on{PreShv}(Y,\on{Ab})^\heartsuit$
is the abelian category of presheaves of abelian groups. The category $\on{Shv}_?(Y,\on{Ab})$ acquires a unique
t-structure for which the sheafification functor is t-exact. 

\medskip

Note, however, that the forgetful functor 
\begin{equation} \label{e:from sheaves to presheaves}
\oblv_?:\on{Shv}_?(Y,\on{Ab})\to \on{PreShv}(Y,\on{Ab})
\end{equation} 
is \emph{not} t-exact, but only left t-exact. 

\sssec{}

The following seems to be a ubiquitous confusion: 

\medskip

Let $\CF$ be an object of 
$\on{Shv}_?(Y,\on{Ab})^\heartsuit$.  We can associate with it two different objects of $\on{PreShv}(Y,\on{Ab})$. One is 
$\oblv_?(\CF)$.  The other is $$\CF^0:=H^0(\oblv_?(\CF))\in \on{PreShv}(Y,\on{Ab})^\heartsuit.$$ The functor $\CF\mapsto \CF^0$
$$\on{Shv}_?(Y,\on{Ab})^\heartsuit\to \on{PreShv}(Y,\on{Ab})^\heartsuit$$
is the ``usual" (non right-exact) embedding of the abelian category of sheaves into the abelian category of presheaves.
The natural map $\CF^0\to \oblv_?(\CF)$ induces an isomorphism of the sheafifications
$$\CF\simeq (\CF^0)_?\simeq (\oblv_?(\CF))_?.$$

\medskip

Note also that 
$$\Gamma(U,\CF):=\Gamma(U,\oblv_?(\CF))\simeq \Gamma_?(U,\CF^0),$$
where the latter is known as the ``sheaf cohomology" of $\CF^0$ over $U$. 

\sssec{}

For a map $f:Y_1\to Y_2$ between schemes we will denote by 
$$f_*:\on{PreShv}(Y_1,\on{Ab})\to \on{PreShv}(Y_2,\on{Ab})$$
the corresponding direct image functor. 

\medskip

This functor sends $\on{Shv}_?(Y_1,\on{Ab})\to \on{Shv}_?(Y_2,\on{Ab})$. We will denote the resulting functor
by the same symbol $f_*$ (and not $Rf_*$, which may be more traditional). 

\ssec{Different ways to view the $K$-theory functor}  \label{ss:K}

\sssec{}

First off, for an integer $n$, we consider the presheaf $K_n$ on the category of affine schemes
$$S\mapsto K_n(S);$$
this is a presheaf of abelian groups. 

\sssec{}

We can sheafify this presheaf on the big Zariski site, and obtain a (pre)sheaf of chain complexes of abelian
groups, denoted $\CK_n$.  Applying the procedure of right Kan extension along
$$(\on{Sch}^{\on{aff}})^{\on{op}}\hookrightarrow \on{PreStk}^{\on{op}}$$
we can evaluate $\CK_n$ on any prestack
$$\CY\mapsto \CK_n(\CY).$$

\sssec{}

The restriction of $\CK_n$ to the small Zariski site of an individual scheme $Y$ will be denoted by 
$$\CK_n|_Y\in \on{Shv}_{\on{Zar}}(Y,\on{Ab})^\heartsuit.$$

We have, by definition,
$$\Gamma(Y,\CK_n|_Y)=\CK_n(Y).$$

Note that $\CK_n|_Y$ is canonically isomorphic to the Zariski sheafification of the restriction of $K_n$ 
to the small Zariski site of $Y$. 

\sssec{}

The functor
$$S\mapsto H^0(\CK_n(S))$$
is an abelian group-object in the category of \emph{discrete} prestacks; we will denote it by $\bK_n$. For
a prestack $\CY$, we have
$$
\pi_i(\Maps(\CY,\bK_n))=
\begin{cases}
& H^0(\CK_n(\CY)) \text{ for } i=0; \\
& 0 \text{ for } i>0.
\end{cases}
$$

\medskip 

For $j\geq 0$, we can consider the $i$-fold classifying space of $\bK_n$, sheafified in Zariski topology,
denoted $B^j_{\on{Zar}}(\bK_n)$. This is a prestack that takes values in the category of 
$E_\infty$-group objects in the category of prestacks. For a prestack $\CY$, we have
$$
\pi_i(\Maps(\CY,B^j_{\on{Zar}}(\bK_n)))=
\begin{cases}
& H^{j-i}(\CK_n(\CY)) \text{ for } i\leq j; \\
& 0 \text{ for } i>j.
\end{cases}
$$

\section{Construction in particular cases}  \label{s:particular}

\ssec{What we wish to construct}  \label{ss:wish}

\sssec{}

We interpret the Brylisnki-Deligne datum as a \emph{based} map
\begin{equation} \label{e:datum}
\kappa:X\times B_{\on{Zar}}(G) \to B^2_{\on{Zar}}(\bK_2).
\end{equation}

\medskip

From this datum we wish to produce a factorization line bundle $\CL_\Gr$ on the affine Grassmannian
$\Gr_{G,\Ran}$. In other words, given an affine test-scheme $S$ and a triple
$$(I,\CP_G,\alpha),$$
where 

\begin{itemize}

\item $I$ is a finite set of maps $S\to G$;

\item $\CP_G$ is a $G$-bundle on $S\times X$;

\item $\alpha$ is a trivialization of $\CP_G$ on the open $U_I\subset S\times X$ equal to the complement of $Z_I$ --
the union of the graphs of the maps that comprise $I$,

\end{itemize}
we wish to construct a line bundle $\CL_{I,\CP_G,\alpha}$ on $S$. This construction is supposed to be functorial in $S$ in the natural
sense. 

\sssec{} 

The factorization structure on the collection
\begin{equation} \label{e:output}
(I,\CP_G,\alpha)\rightsquigarrow \CL_{I,\CP_G,\alpha}
\end{equation}
means the following:

\medskip

Let us be given a partition $I=I^1\sqcup I^2$, such that for any $i^1\in I^1$ and $i^2\in I^2$, the corresponding two maps
$$S\rightrightarrows X$$
have non-intersecting graphs. The factorization structure on the affine Grassmannian means that the datum of $(I,\CP_G,\alpha)$
with a given $I$ is equivalent to the data of $(I^1,\CP^1_G,\alpha^1)$ and $(I^2,\CP^2_G,\alpha^2)$.

\medskip

In this case, we wish to have an isomorphism of line bundles on $S$:
$$\CL_{I,\CP_G,\alpha}\simeq \CL_{I^1,\CP^1_G,\alpha^1}\otimes \CL_{I^2,\CP^2_G,\alpha^2}.$$

\medskip

These isomorphisms are supposed to be compatible with further partitions in the natural sense. 

\ssec{The initial attempt}

\sssec{}

The first glitch comes from the fact that the datum of $\CP_G$ does \emph{not} define a map $S\times X\to B_{\on{Zar}}(G)$,
but rather a map $S\times X\to B_{\on{et}}(G)$. However, this is easily overcome using \cite[Theorem 2]{DS}: namely, $\CP_G$
\emph{does} give a map $S\times X\to B_{\on{Zar}}(G)$ after passing to an \'etale cover of $S$. 

\medskip

Composing with the map $\kappa$ of \eqref{e:datum}, we obtain a map
$$S\times X\to B^2_{\on{Zar}}(\bK_2)$$
that vanishes when restricted to $U_I$. 

\sssec{}  \label{sss:kappa S}

Let us denote by 
$$U_I\overset{j}\to S\times X\overset{\iota}\leftarrow Z_I$$
the corresponding maps. 

\medskip

Thus, we obtain a section $\kappa_S$ of 
$$\iota^!(\CK_2|_{S\times X})[2]\in \on{Shv}(Z_I,\on{Ab}).$$

\medskip

Consider the projection
$$\pi:S\times X\to S.$$

If we had a residue map
\begin{equation} \label{e:residue}
(\pi\circ \iota)_*(\iota^!(\CK_2|_{S\times X}))[2]\to \CK_1|_S[1],
\end{equation} 
then from $\kappa_S$ we would obtain a section of 
$$\CK_1|_S[1] \simeq \CO_S^\times[1],$$
i.e., a line bundle on $S$. 

\sssec{}

The problem is, however, that we were \emph{not} able to define the map \eqref{e:residue} in general.

\medskip

In the rest of this section we will explain two particular cases in which we can construct 
such a map.  In \secref{s:spectrum} we will construct a variant of the map \eqref{e:residue},
where instead of $K_2$ we use the 2-truncated K-theory spectrum. In Sects \ref{s:construction abs} 
and \ref{s:construction gen} we will combine 
all these particular cases to produce the desired system of 
line bundles \eqref{e:output}. 

\ssec{The smooth case}   \label{ss:smooth}

Let us assume that $S$ is smooth. In this case, the map \eqref{e:residue} is easy to construct, using Gersten resolution.

\sssec{}  \label{sss:Gersten}

Let us recall that for a regular scheme $Y$, the object 
$$\CK_n|_Y\in \on{Shv}_{\on{Zar}}(Y,\on{Ab})$$ can be represented by a complex (called the Gersten resolution) 
\begin{equation} \label{e:Gersten}
\on{C}^\cdot_n:=\on{C}^0_n\to \on{C}^1_n\to...\to \on{C}^n_n,
\end{equation} 
where $$\on{C}^i_n\simeq \underset{Y_i}\oplus\, (j_{Y_i})_*(K_{n-i}(\eta_{Y_i})),$$
where the direct sum is taken over irreducible subvarieties of $Y$ of condimension $i$ and   
$$\eta_{Y_i}\overset{j_{Y_i}}\hookrightarrow Y$$
denotes the inclusion of the generic point. 

\sssec{}

Thus, we obtain that in the situation of \secref{sss:kappa S}, the object $\iota^!(\CK_2|_{S\times X})[1]$
is represented by the complex
$$\underset{S'}\oplus\, (j_{S'})_*(\CO^\times_{\eta_{S'}}) \to \underset{v'}\oplus \,\BZ,$$
where $S'$ runs over the set of irreducible components of $Z_I$ and $v'$ run over the set of points of codimension $1$ in $Z_I$.

\medskip

Note that each $S'$ as above is finite over $S$. From here we obtain that the norm map defines a map of complexes:
\begin{equation} \label{e:kappa S smooth local}
\CD
\underset{S'}\oplus\, (j_{S'})_*(\CO^\times_{\eta_{S'}}) @>>>   \underset{v'}\oplus \,\BZ   \\
@VVV   @VVV  \\
\CO^\times_{\eta_{S}} @>>>  \underset{v}\oplus \,\BZ,
\endCD
\end{equation} 
where $\eta_S$ is the generic point of $S$ (we are assuming for simplicity that $S$ is connected), and $v$ runs over
the set of points of codimension $1$ in $S$. Now, the complex 
$$\CO^\times_{\eta_{S}} \to  \underset{v}\oplus \,\BZ$$
is a resolution of $\CO^\times_S$ (in fact, it is a particular case of \eqref{e:Gersten} for $n=1$). 

\sssec{}

Thus, the map \eqref{e:kappa S smooth local} gives rise to a map
\begin{equation} \label{e:dir im residue}
(\pi\circ \iota)_*(\iota^!(\CK_2|_{S\times X}))[2]\to \CK_1|_S[1],
\end{equation} 
and thus, does produce from $\kappa_S$ a section of $\CO_S^\times[1]$,
as desired. 

\medskip

Moreover, it is easy to see that this construction is functorial, and has the required factorization structure. 

\ssec{The global case}  \label{ss:global}

In this subsection we will assume that $X$ is complete. We will show that in this case, the system of line bundles
\eqref{e:output} can be constructed, albeit so far without the factorization structure.

\sssec{}

Consider the algebraic stack $\Bun_G(X)$. We will show that the datum of $\kappa$ gives rise to a line bundle 
$\CL_{\Bun}$ on $\Bun_G(X)$. The sought-for line bundle $\CL_{\Gr}$ on $\Gr_{G,\Ran}$ will then be obtained by pullback.

\medskip

The datum of $\CL_{\Bun}$ is equivalent to the datum of functorial assignment to any affine scheme $S$ and a $G$-bundle 
$\CP_G$ on $S\times X$ of a line bundle $\CL_{\CP_G}$ on $S$. 

\medskip

Since $\Bun_G(X)$ is a \emph{smooth} algebraic stack, by \'etale descent for line bundles on $S$, 
it suffices to consider the case when $S$ itself is smooth. After further \'etale localization 
with respect to $S$, we can assume that $\CP_G$ comes from a map
$$S\times X\to B_{\on{Zar}}(G).$$

\medskip

Composing with $\kappa$ we obtain a map
$$S\times X\to B^2_{\on{Zar}}(\bK_2),$$
i.e., a section of $\CK_2|_{S\times X}[2]$.

\medskip

To this datum we will associate a section of $\CK_1|_S[1]\simeq \CO^\times_S[1]$, i.e.,
a line bundle on $S$. This will be a variant of the construction in \secref{ss:smooth}. Namely, we will construct 
a map 
\begin{equation} \label{e:dir im residue glob}
\pi_*(\CK_2|_{S\times X})[2]\to \CK_1[1]|_S.
\end{equation} 

\sssec{}

We need construct a map 
\begin{equation} \label{e:dir im residue glob again}
\pi_*(\CK_2|_{S\times X})[1]\to \CO^\times_S.
\end{equation} 

\medskip

Since $S$ was assumed smooth, it is enough to construct it outside of codimension 2 in $S$. Hence, there are two
cases to consider: when $S$ is the spectrum of a field and when $S$ is a spectrum of a DVR.

\sssec{}

Let first $S$ be the spectrum of a field $F$. We need to construct a map
$$H^1(X_F,\CK_2|_{X_F}) \to F^\times.$$

This is well-known, using the Gersten complex on $X_F$. Indeed, we have a map 
$$\underset{v}\oplus\, F_v^\times \to F^\times$$
(where the direct sum is taken over the set of closed points $v$ of $X_F$ and we denote by $F_v$ the residue
field at $v$), given by the norm maps $\on{Norm}_{F_v/F}:F_v^\times \to F^\times$.

\medskip

Now, the composite 
$$K_2(\eta_{X_F})\to \underset{v}\oplus\, F_v^\times \to F^\times$$
vanishes by Weil's product formula. 

\sssec{}

Let now $S=\Spec(R)$ be the spectrum of a DVR with generic point $\eta=\Spec(F)$ and a closed point $s$. We need to show that the composite map
$$H^1(S\times X,\CK_2|_{S\times X})\to H^1(X_F,\CK_2|_{X_F}) \to F^\times$$
factors through $R^\times\subset F^\times$. Equivalently, we need to show that the composition
$$H^1(S\times X,\CK_2|_{S\times X})\to H^1(X_F,\CK_2|_{X_F}) \to F^\times \to \BZ,$$
where the last map is given by the valuation at $s$, vanishes. In doing so, we can replace $R$ by its completion. 

\medskip

Let $v$ be a closed point of $X_F$. Taking its normalization in $S\times X$, it corresponds to a DVR, denoted $R'$, such that
$S'=\Spec(R')$ is equipped with a finite flat map $S'\to S$. The required assertion follows, using the Gersten complex 
computing $H^1(S\times X,\CK_2|_{S\times X})$, from the commutative diagram
$$
\CD
F_v^\times @>>>  \BZ  \\
@V{\on{Norm}_{F_v/F}}VV   @VV{\on{deg}}V   \\
F^\times @>>>  \BZ  
\endCD
$$
where the right vertical map is given by multiplication by the degree of the residue field extension. 

\sssec{}

It is easy to see that the above construction is compatible with the one in \secref{ss:smooth}. I.e., if
$\CP_G$ is the second component of a triple $(I,\CP_G,\alpha)$, then the resulting two line bundles 
on $S$ are canonically isomorphic. 

\section{Construction using the $K$-theory spectrum}  \label{s:spectrum}

In this section we will now produce a variant of the map \eqref{e:residue} in a different context, where instead of
$\CK_2(-)$ we use the entire $K$-theory spectrum. 

\ssec{The K-theory spectrum as a sheaf}

\sssec{}

First, we note that the discussion in \secref{ss:sheaves} applies verbatim, where instead of the category 
$\on{PreShv}(Y,\on{Ab})$ we consider $\on{PreShv}(Y,\on{Sptr})$, the category of presheaves of spectra.

\sssec{} 

Recall that to any scheme $Y$ (assumed quasi-separated and quasi-compact) 
we can attach the \emph{non-connective} K-theory spectrum of the category $\on{Perf}(Y)$,
denoted $K_{\on{nc}}(Y)$, see \cite[Sect. 6]{TT}. For $i\geq 0$, the $i$-th homotopy group $\pi_i(K_{\on{nc}}(Y))$ is the usual $K_i(Y)$. 

\medskip

Recall also that if $Y$ is regular, the spectrum $K_{\on{nc}}(Y)$ is actually connective, i.e., its 
negative homotopy groups vanish.

\sssec{}

The assignment
$$Y\mapsto K_{\on{nc}}(Y)$$
is functorial, and thus defines a presheaf of spectra on the category of schemes.

\medskip

The key property of $K_{\on{nc}}$ that we will use is that it is a \emph{Zariski sheaf}, see \cite[Theorem 8.1]{TT}.
We will also use the notation
$$\CK_{\on{nc}}(Y):=K_{\on{nc}}(Y).$$

\medskip

Applying the procedure of right Kan extension along
$$(\on{Sch}^{\on{aff}})^{\on{op}}\hookrightarrow \on{PreStk}^{\on{op}},$$
we can evaluate $\CK_{\on{nc}}$ on any prestack
$$\CY\mapsto \CK_{\on{nc}}(\CY).$$

\medskip

Note, however, that the Zariski descent property of $\CK_{\on{nc}}$ implies that if $\CY=Y$ is a scheme, 
the two definitions of $\CK_{\on{nc}}(Y)$ are consistent.

\sssec{}

For an individual scheme $Y$, we will denote by 
$$\CK_{\on{nc}}|_Y\in \on{PreShv}(Y,\on{Sptr})$$ 
the restriction of $\CK_{\on{nc}}$ to the small Zariski site of $Y$. 

\medskip

We will also use its
truncations 
$$\tau^{\leq n}(\CK_{\on{nc}}|_Y) \text{ and } \tau^{\geq n}(\CK_{\on{nc}}|_Y)\in \on{Shv}_{\on{Zar}}(Y,\on{Sptr}),$$
with respect to the t-structure on $\on{Shv}_{\on{Zar}}(Y,\on{Sptr})$.

\medskip

For $n\geq 0$, we have
$$H^{-n}(\CK_{\on{nc}}|_Y)\simeq \CK_n|_Y.$$

\sssec{}

Let $\bK_{\on{nc}}$, $\bK_{\geq n}$ and $\bK_{\leq -n}$ denote the $E_\infty$-objects in the category of prestacks defined
by
$$\Maps(S,\bK_{\on{nc}}):=\Omega^\infty(\CK_{\on{nc}}(S))=\Omega^\infty\left(\Gamma(S,\CK_{\on{nc}}|_S)\right)$$
and
$$\Maps(S,\bK_{\geq n}):=\Omega^\infty\left(\Gamma(S,\tau^{\leq -n}(\CK_{\on{nc}}|_S))\right) \text{ and }
\Maps(S,\bK_{\leq n}):=\Omega^\infty\left(\Gamma(S,\tau^{\geq -n}(\CK_{\on{nc}}|_S))\right).$$

Due to Zariski descent, the same isomorphisms hold for $S$ replaced by any scheme $Y$.

\sssec{}

We have the canonical maps
$$\bK_{\on{nc}} \leftarrow \bK_{\geq n} \to B^n_{\on{Zar}}(\bK_n).$$

\ssec{Construction of the trace map}  \label{ss:spectrum}

\sssec{}

Consider $\CK_{\on{nc}}|_{S\times X}$, and consider its $2$-connective truncation
\begin{equation} \label{e:truncation}
\tau^{\leq -2}(\CK_{\on{nc}}|_{S\times X})\in \Shv_{\on{Zar}}(S\times X,\on{Sptr}).
\end{equation} 

\medskip

We will now construct a trace map
\begin{equation} \label{e:residue spectrum}
\tau^{\leq 0}\left((\pi\circ \iota)_*(\iota^!(\tau^{\leq -2}(\CK_{\on{nc}}|_{S\times X})))\right)\to \CK_1|_S[1]. 
\end{equation} 

\sssec{}  \label{sss:construction of residue}

For every open $V\subset S\times X$ we have a triangle of categories
$$\on{Perf}(V)_{V\cap Z_I}\to \on{Perf}(V)\to \on{Perf}(V\cap U_I),$$
where $\on{Perf}(V)_{V\cap Z_I}\subset \on{Perf}(V)$ is the full subcategory 
spanned by objects with \emph{set-theoretic support} on $V\cap Z_I\subset Z_I$. 

\medskip

We have the corresponding fiber sequence of spectra (see \cite[Theorem 7.4]{TT})
\begin{equation} \label{e:fiber sequence}
K_{\on{nc}}(V)_{Z_I}\to K_{\on{nc}}(V)\to K_{\on{nc}}(V\cap U_I),
\end{equation} 
where $K_{\on{nc}}(V)_{Z_I}$ is the non-connective K-theory spectrum of the category 
$\on{Perf}(V)_{V\cap Z_I}$. 

\sssec{}

The assignment 
$$V\mapsto K_{\on{nc}}(V)_{Z_I}$$
defines a sheaf of spectra in the Zariski topology on $S\times X$; denote it by $(\CK_{\on{nc}}|_{S\times X})_{Z_I}$. 

\medskip

From \eqref{e:fiber sequence}
we obtain
\begin{equation} \label{e:with support}
(\CK_{\on{nc}}|_{S\times X})_{Z_I}\simeq \iota_*\circ \iota^!(\CK_{\on{nc}}|_{S\times X}).
\end{equation} 

\sssec{}

Direct image along $\pi$ gives rise to a functor 
$$\on{Perf}(S\times X)_{Z_I}\to \on{Perf}(S)$$
(also for $S$ replaced by its Zariski open subsets).  From here we obtain a map in $\Shv(S,\on{Sptr})$
$$\pi_*((\CK_{\on{nc}}|_{S\times X})_{Z_I})\to \CK_{\on{nc}}|_S,$$
and composing with \eqref{e:with support}, we obtain a map
\begin{equation} \label{e:res all K}
(\pi\circ \iota)_*(\iota^!(\CK_{\on{nc}}|_{S\times X}))\to \CK_{\on{nc}}|_S.
\end{equation} 

\begin{lem}  \label{l:triv to K0}
The composition
\begin{multline*}
\tau^{\leq 0}\left((\pi\circ \iota)_*(\iota^!(\tau^{\leq -2}(\CK_{\on{nc}}|_{S\times X})))\right)\to 
(\pi\circ \iota)_*(\iota^!(\tau^{\leq -2}(\CK_{\on{nc}}|_{S\times X})))\to \\
\to (\pi\circ \iota)_*(\iota^!(\CK_{\on{nc}}|_{S\times X}))\overset{\text{\eqref{e:res all K}}}\longrightarrow \CK_{\on{nc}}|_S
\end{multline*}
(uniquely) factors via a map 
\begin{equation} \label{e:res all K estim}
\tau^{\leq 0}\left((\pi\circ \iota)_*(\iota^!(\tau^{\leq -2}(\CK_{\on{nc}}|_{S\times X})))\right)\to  \tau^{\leq -1}(\CK_{\on{nc}}|_S).
\end{equation} 
\end{lem}

\begin{proof}

The assertion of the lemma is equivalent to saying that 
the map
$$H^0\left((\pi\circ \iota)_*(\iota^!(\tau^{\leq -2}(\CK_{\on{nc}}|_{S\times X})))\right)\to  H^0(\CK_{\on{nc}}|_S)\simeq \CK_0|_S.$$
vanishes. 

\medskip

Note that $\CK_0|_S$ is the constant sheaf with stalk $\BZ$. Hence, it is enough to show that the pullback
of the above composition to any closed point of $S$ vanishes. This allows to replace $S$ by $\on{pt}=\Spec(k)$. 

\medskip

However, in the latter case, we claim that the object 
$$\Gamma\left(X,\iota_*\circ \iota^!(\tau^{\leq -2}(\CK_{\on{nc}}|_X))\right)$$
lives in cohomological degrees $\leq -1$. Indeed, this follows from the fact that both functors
 $$\Gamma(X,-) \text{ and } \Gamma(U_I,-):\on{Shv}_{\on{Zar}}(X,\on{Ab})\to \on{Ab}$$
have cohomological amplitude bounded by $\dim(X)=1$. 

\end{proof} 

\sssec{}

Thus, composing with the projection
$$\tau^{\leq -1}(\CK_{\on{nc}}|_S)\to \CK_1|_S[1],$$
from \eqref{e:res all K estim}, we obtain the desired map \eqref{e:residue spectrum}.

\ssec{Construction of the line bundle from a map to the K-theory spectrum}  \label{ss:constr from full}

\sssec{}

Recall the prestack $\bK_{\geq 2}$, which is equipped with a map
\begin{equation} \label{e:truncation to all}
\bK_{\geq 2}\to B^2_{\on{Zar}}(\bK_2).
\end{equation} 

\medskip

In particular, instead of talking about a Brylinski-Deligne datum, which was a based map 
$$\kappa:X\times B_{\on{Zar}}(G) \to B^2_{\on{Zar}}(\bK_2),$$
we can talk about a datum of a based map
\begin{equation} \label{e:refined datum}
\kappa':X\times B_{\on{Zar}}(G)\to \bK_{\geq 2}
\end{equation} 
(as usual, ``based" means that we consider maps whose precomposition with $X\times \on{pt}\to X\times B_{\on{Zar}}(G)$
is trivialized). 

\sssec{}  \label{sss:constr from full}

We claim that a datum of such a $\kappa'$ does give rise to a factorizable line bundle on $\Gr_{G,\Ran}$. 

\medskip

Indeed, given a point $(I,\CP_G,\alpha)$ of $\Gr_{G,\Ran}$, precomposing with $\kappa'$ we obtain a section $\kappa'_S$  of
$\tau^{\leq 0}\left((\pi\circ \iota)_*(\iota^!(\tau^{\leq -2}(\CK_{\on{nc}}|_{S\times X})))\right)$, and using 
\eqref{e:residue spectrum}, we obtain from $\kappa'_S$ a section of $\CK_1|_S[1]\simeq \CO^\times_S[1]$, i.e., a line bundle on $S$.

\medskip

Moreover, this construction is functorial in $S$, and has the required factorization structure. 

\ssec{Compatibility for smooth schemes}

In this subsection we will assume that $S$ is smooth. 

\sssec{}

We claim that the constructions in \secref{ss:constr from full} and \secref{ss:smooth}
are compatible in the following sense:

\medskip

Start with a datum of $\kappa'$, and let $\kappa$ be its projection with respect to the map \eqref{e:truncation to all}. 
Then we claim that the resulting line bundles on $S$ for every $(I,\CP_G,\alpha)$ are canonically isomorphic. 
Indeed, this follows from the next observation: 

\begin{lem}  \label{l:full spectrum compat for smooth}
Assume that $S$ is smooth (in particular, $\CK_{\on{nc}}|_{S\times X}$ is connective). In this case, 
we have a commutative diagram
$$
\CD
(\pi\circ \iota)_*(\iota^!(\tau^{\leq -2}(\CK_{\on{nc}}|_{S\times X})))  @>>> (\pi\circ \iota)_*(\iota^!(\CK_{\on{nc}}|_{S\times X})) 
@>{\text{\eqref{e:res all K}}}>> \CK_{\on{nc}}|_S \\
@VVV   & & @VVV  \\
(\pi\circ \iota)_*(\iota^!(\CK_2|_{S\times X})[2]) @>{\text{\eqref{e:dir im residue}}}>> 
\CK_1|_S[1]  @>>>  \tau^{\geq -1}(\CK_{\on{nc}}|_S). 
\endCD
$$
\end{lem} 

The proof is obtained by unwinding the definitions. 

\sssec{}   

Assume now also that $X$ is complete. The construction in \secref{ss:spectrum} has a variant, where we now use the direct 
image functor
$$\pi_*:\on{Perf}(S\times X)\to \on{Perf}(S)$$
and the corresponding trace map 
\begin{equation} \label{e:res all K glob}
\pi_*(\CK_{\on{nc}}|_{S\times X}) \to \CK_{\on{nc}}|_S.
\end{equation}

\medskip

\begin{lem}  \label{l:trivial estimate}
For $i\geq 2$, the object
$$\pi_*(\CK_i|_{S\times X})\in \Shv_{\on{Zar}}(S,\on{Ab})$$ is concentrated in cohomological degrees $\leq i-1$.
\end{lem} 

For the proof see \secref{sss:proof of vanish}.  From \lemref{l:trivial estimate}, we obtain that for $i\geq 2$, the object 
$\pi_*(\tau^{\leq -i}(\CK_{\on{nc}}|_{S\times X}))$  
is concentrated in cohomological degrees $\leq -1$.

\begin{rem}  \label{r:vanishing}
As was explained to us by experts, the object $\pi_*(\CK_i(S\times X))$ is actually supposed to be 
concentrated in cohomological degrees $\leq 1$. A proof outlined for us by S.~Bloch and H.~Esnault
uses a version of the \emph{moving lemma}, which does not seem to be present in the literature. 

\medskip

Note that this vanishing would imply that $\pi_*(\tau^{\leq -i}(\CK_{\on{nc}}|_{S\times X}))$  
is concentrated in cohomological degrees $\leq -(i-1)$. Using this result would have simplified
the proof of \propref{p:glob vs fat} below. 
\end{rem}

\sssec{}  

From \eqref{e:res all K glob} we obtain a map
\begin{equation} \label{e:residue spectrum global all}
\pi_*(\tau^{\leq -2}(\CK_{\on{nc}}|_{S\times X}))\to \tau^{\leq -1}(\CK_{\on{nc}}|_S),
\end{equation} 
which we can further compose with the projection $\tau^{\leq -1}(\CK_{\on{nc}}|_S)\to \CK_1|_S[1]$
and obtain a map
\begin{equation} \label{e:residue spectrum global}
\pi_*(\tau^{\leq -2}(\CK_{\on{nc}}|_{S\times X}))\to \CK_1|_S[1].
\end{equation}

\medskip

Hence, a datum of $\kappa'$ as in \eqref{e:refined datum} gives rise to a line bundle $\CL'_{\Bun}$ on $\Bun_G(X)$, so that the line bundle constructed
in \secref{ss:constr from full} identifies with the pull-back of $\CL'_{\Bun}$ along the map $\Gr_{G,\Ran}\to \Bun_G(X)$. 

\sssec{}  \label{sss:compat with global}

We claim that the above line bundle $\CL'_{\Bun}$ identifies canonically with the line bundle $\CL_{\Bun}$ constructed
in \secref{ss:global} from $\kappa$ corresponding to the image of $\kappa'$ under the projection \eqref{e:truncation to all}. This 
follows from the following analog of \lemref{l:full spectrum compat for smooth}:

\begin{lem}  \label{l:estimate global smooth}
We have a commutative diagram
$$
\CD
\pi_*(\tau^{\leq -2}(\CK_{\on{nc}}|_{S\times X}))  @>{\text{\eqref{e:residue spectrum global all}}}>>  \tau^{\leq -1}(\CK_{\on{nc}}|_S) \\
@VVV   @VVV   \\
\pi_*(\CK_2|_{S\times X})[2]) @>{\text{\eqref{e:dir im residue glob}}}>>   \CK_1|_S[1].
\endCD
$$
\end{lem} 

\ssec{The case of a fat point}  \label{ss:Artin}

In this subsection we will consider the case when $S$ is the spectrum of a local Artinian ring, i.e., $S$ is a fat point. 
For such $S$, we will be able to construct the map \eqref{e:residue}. 

\sssec{}   \label{sss:less 3}

Consider the exact triangle 
$$\tau^{\leq -3}(\CK_{\on{nc}}|_{S\times X})\to \tau^{\leq -2}(\CK_{\on{nc}}|_{S\times X})\to \CK_2|_{S\times X}[2],$$
and the corresponding triangle
$$(\pi\circ \iota)_*\left(\iota^!(\tau^{\leq -3}(\CK_{\on{nc}}|_{S\times X}))\right) \to 
(\pi\circ \iota)_*\left(\iota^!(\tau^{\leq -2}(\CK_{\on{nc}}|_{S\times X}))\right) \to 
(\pi\circ \iota)_*\left(\iota^!(\CK_2|_{S\times X})\right)[2].$$

We will presently show that $(\pi\circ \iota)_*\left(\iota^!(\tau^{\leq -i}(\CK_{\on{nc}}|_{S\times X}))\right)$ is concentrated 
in cohomological degrees $\leq -(i-1)$. This would imply that the map \eqref{e:residue spectrum} uniquely extends to the
desired map \eqref{e:residue}. 

\sssec{}

To prove the desired cohomological estimate, we notice that the scheme $Z_I$ is Artinian, and therefore
the functor $(\pi\circ \iota)_*$ is t-exact. Hence, it remains to show that the cohomological amplitude of the
functor 
$$\iota^!:\Shv_{\on{Zar}}(S\times X,\on{Ab})\to \Shv_{\on{Zar}}(Z_I,\on{Ab})$$
is bounded on the right by $1$. The latter is equivalent to the functor
$$j_*:\Shv_{\on{Zar}}(U_I,\on{Ab})\to \Shv_{\on{Zar}}(S\times X,\on{Ab})$$
being t-exact. However, the latter holds, because for any point $z\in Z_I$, the scheme $(S\times X)_x\underset{S\times X}\times U_I$
is local. 

\sssec{}

Thus, starting from a datum of $\kappa$, to $S$ as above and any $(I,\CP_G,\alpha)$ we can associate a line bundle on $S$. 
This construction is functorial in $S$ and has the required factorization structure. 

\sssec{}  \label{sss:reduced point}

Assume for a moment that $S$ is reduced (i.e., $S=\on{pt}$). Then, on the one hand, the above construction attaches to $\kappa$
a line (=line bundle over $S$). On the other hand, the construction of \secref{ss:smooth} also produces a a line (=line bundle over $S$).

\medskip

However, it is easy to see from \lemref{l:full spectrum compat for smooth} that these two lines are canonically isomorphic,
in a way compatible with factorization. 

\sssec{}

For the sequel, we will need the following compatibility result between the above construction and the one in \secref{ss:global}: 

\begin{prop}  \label{p:glob vs fat}
Let $X$ be complete. If $S$ is a fat point, and given a triple $(I,\CP_G,\alpha)$, we have a canonical isomorphism of line bundles on $S$
$$\CL_{I,\CP_G,\alpha}\simeq (\CL_{\Bun})|_S,$$
where the map $S\to \Bun_G(X)$ is given by the datum of $\CP_G$. 
\end{prop}

\ssec{Proof of \propref{p:glob vs fat}}

\sssec{}

Since $S\times X$ is of Krull dimension $1$, for any $i$ the object 
$$\pi_*(\tau^{\leq -i}(\CK_{\on{nc}}|_{S\times X}))$$
lives in cohomological degrees $\leq -(i-1)$.

\medskip

Hence, the construction in \secref{sss:less 3} has a variant that shows that the map
\eqref{e:residue spectrum global} extends uniquely to a map 
\begin{equation} \label{e:dir im residue glob fat}
\pi_*(\CK_2|_{S\times X})[2]\to \CK_1|_S[1].
\end{equation}

This construction is compatible with one in \secref{sss:less 3} via the commutative diagram
$$
\CD
(\pi\circ \iota)_*\left(\iota^!(\tau^{\leq -2}(\CK_{\on{nc}}|_{S\times X}))\right)   @>>>  \pi_*(\tau^{\leq -2}(\CK_{\on{nc}}|_{S\times X}))  \\
@VVV     @VVV    \\
(\pi\circ \iota)_*\left(\iota^!(\CK_2|_{S\times X})\right)[2]  @>>>    \pi_*(\CK_2|_{S\times X})[2]   \\
@VVV    @VVV    \\
\CK_1|_S[1]  @>{\on{Id}}>> \CK_1|_S[1].
\endCD
$$ 

\medskip

The given $S$-point of $\Bun_G(X)$ can be factored as
$$S\overset{f}\to S'\to \Bun_G(X)$$
with $S'$ smooth. Let us denote by $\pi'$ the projection $S'\times X\to S'$. 

\medskip

Hence, in order to prove the proposition, we need to establish the commutativity of the diagram
\begin{equation} \label{e:diag to commute}
\CD
f^*(\pi'_*(\CK_2|_{S'\times X}))[2]  @>>>  \pi_*(\CK_2|_{S\times X})[2]    \\
@V{\text{\eqref{e:dir im residue glob}}}VV   @VV{\text{\eqref{e:dir im residue glob fat}}}V   \\
f^*(\CK_1|_{S'}[1])   @ >>>      \CK_1|_S[1].
\endCD    
\end{equation} 

\begin{rem}
The above commutativity is not tautological because the maps \eqref{e:dir im residue glob} and 
\eqref{e:dir im residue glob fat} were constructed by different procedures.
\end{rem}

\sssec{}

Note that in the diagram 
$$
\CD
f^*(\pi'_*(\tau^{\leq -2}(\CK_{\on{nc}}|_{S'\times X})))   @>>>  \pi_*(\tau^{\leq -2}(\CK_{\on{nc}}|_{S\times X}))  \\
@VVV   @VVV  \\
f^*(\pi'_*(\CK_2|_{S'\times X}))[2]  @>>>  \pi_*(\CK_2|_{S\times X})[2]    \\
@V{\text{\eqref{e:dir im residue glob}}}VV   @VV{\text{\eqref{e:dir im residue glob fat}}}V   \\
f^*(\CK_1|_{S'}[1])   @ >>>      \CK_1|_S[1].
\endCD    
$$
the outer square is commutative (by \lemref{l:estimate global smooth}).  

\medskip

Therefore, if we we knew that
$\pi'_*(\tau^{\leq -3}(\CK_{\on{nc}}|_{S'\times X}))$ lives in cohomological degrees $\leq -2$ 
(see Remark \ref{r:vanishing}) we would be done. If we want to avoid using this result,
we argue as follows.

\sssec{}

The possible obstruction to the commutativity of \eqref{e:diag to commute} is a map
$$f^*(\pi'_*(\tau^{\leq -3}(\CK_{\on{nc}}|_{S'\times X})))\to \CK_1|_S[1],$$
i.e., a map
\begin{equation} \label{e:obstruction}
H^0\left(f^*(\pi'_*(\tau^{\leq -3}(\CK_{\on{nc}}|_{S'\times X})))[-1]\right) \to \CO_S^\times,
\end{equation} 
where we know that $f^*(\pi'_*(\tau^{\leq -3}(\CK_{\on{nc}}|_{S'\times X})))[-1]$ lives in degrees $\leq 0$.

\medskip

The map \eqref{e:obstruction} does vanish when further composed with $\CO^\times_S\to \CO^\times_{S_{\on{red}}}$ by
\secref{sss:reduced point}. Hence, arguing by induction on the length of $S$, we can assume that $S$ is is the scheme
of dual numbers, i.e., is isomorphic to $\Spec(k[\epsilon])$, where $\epsilon^2=0$. In this case, the map $f:S\to S'$
can be further factored as
$$S\to S''\to S',$$
where $S''$ is smooth of dimension $1$.  Hence, in checking the commutativity of \eqref{e:diag to commute}
we can replace $S'$ by $S''$. So we can assume that $\dim(S')=1$. 

\medskip

However, we claim that for $\dim(S')=1$, the statement that $\pi'_*(\tau^{\leq -3}(\CK_{\on{nc}}|_{S'\times X}))$ lives in cohomological degrees $\leq -2$
holds unconditionally. Indeed, $\pi'_*(\CK_i|_{S'\times X})$ lives in degrees $\leq 2$ since the Krull dimension of $S'\times X$ equals $2$.
So it suffices to show that the $2$-nd cohomology of $\pi'_*(\CK_3|_{S'\times X})$ vanishes. 

\sssec{}

To prove the required vanishing, we can replace $S'$ by its localization at the closed point. Denote by $s'$ (resp., $\eta'$) the closed
(resp., generic) point of $S'$. Clearly, $H^2(\pi'_*(\CK_3|_{S'\times X}))$ vanishes over $\eta'$. Its stalk at $s'$ is isomorphic to
$$H^2(S'\times X,\CK_3|_{S'\times X}).$$

We use the Gersten resolution (see \secref{sss:Gersten}) to compute this cohomology. A degree $2$ cocycle is represented by
$$a_i\cdot (s'\times x_i), \quad x_i\in X(k), \quad a_i\in k^\times=K_1(k).$$

However, such a cocycle is always a coboundary, namely, of the element
$$\{t,a_i\}\cdot (\eta'\times x_i),$$
where $t$ is a uniformizer of $S'$, and $\{t,a_i\}$ denotes the product of $t$ and $a_i$, viewed as an element of $K_2$ at the generic
point of the codimension one subscheme $S'\times x_i\subset S'\times X$.

\qed

\sssec{Proof of \lemref{l:trivial estimate}}   \label{sss:proof of vanish}

We may assume that $S$ is local, and we need to show that 
$$\Gamma(S\times X, \CK_i)$$
vanishes in degrees $\geq i$. First off, the Gersten resolution shows that this cohomology vanishes in degrees $\geq i+1$,
so we need to show the vanishing of the $i$-th cohomology. 

\medskip

The group of $i$-cochains in the Gersten complex computing $\Gamma(S\times X, \CK_i)$ is spanned by classes of
codimension-$i$ cycles of two kinds.

\medskip

The first kind is cycles of the form $Z\times X$, where $Z$ is an irreducible subvariety of codimension $i$ in $S$.
The other kind is of the form $Z'$, where $Z'\subset S\times X$ is an irreducible subvariety of codimension $i$ in $S\times X$
that projects in a finite way to some $Z\subset X$. Let us show that both kinds are rationally trivial.

\medskip

For the first kind, it suffices to show that $Z$ is rationally trivial in $S$. This is obvious: replace $S$ be the localization $\wt{S}$
at the generic point of $Z$, so that $\wt{Z}:=\wt{S}\underset{S}\times Z$ becomes the closed point of $\wt{S}$; choose 
any regular one-dimensional closed subscheme $T\subset \wt{S}$, and consider the uniformizer on $T$.

\medskip

For the second kind, by replacing $S$ by $T$ as above, we can assume that $S$ is one-dimensional, and
$Z'$ is a closed point in the special fiber. We need to show that such $Z'$ is rationally trivial. We can find an irreducible 
closed subscheme $S'\subset S\times X$, whose intersection with
the special fiber contains $Z'$ as a connected component. The projection of $S'$ onto $S$ is finite; in particular $S'$
is semi-local. Hence, we can find a regular function on $S'$ whose subscheme of zeros is $Z'$. 

\qed

\section{Construction in the absolute case}  \label{s:construction abs}

In this section we will assume that $(\on{char}(k),N)=1$, where $N$ is the integer from \secref{sss:integer N}. 

\medskip

We will construct the line bundle $\CL_{\Gr}$ associated to an \emph{absolute} Brylinski-Deligne datum, i.e., a map
$$\kappa:B_{\on{Zar}}(G)\to B^2_{\on{Zar}}(\bK_2).$$

\ssec{The raising to the power trick}  \label{ss:power} 

The ``trick" used in this subsection was suggested to us by A.~Beilinson. 

\sssec{}

With no restriction of generality we can assume that our curve $X$ is complete. According to \secref{ss:global}, to a datum of $\kappa$ 
we can attach a line bundle $\CL_{\Bun}$ on $\Bun_G(X)$.
We let $\CL_{\Gr}$ denote its pullback onto $\Gr_{G,\Ran}$. Our task is to endow $\CL_{\Gr}$ with a factorization structure. 

\medskip

According to \propref{p:glob vs fat}, when we evaluate $\CL_{\Gr}$ on
$$S\to \Gr_{G,\Ran}$$
with $S$ being a fat point, it does possess the required factorization structure. 
According to \secref{ss:smooth} the same holds for $S$ being a smooth scheme. 

\medskip

We will now show that this factorization structure \emph{uniquely} extends to any $S$.

\sssec{}  \label{sss:uniqueness}

First off, the uniqueness statement is clear. Indeed, any two isomorphisms between two given line bundles on
a scheme $S$ coincide if they coincide when pulled back along any $S'\to S$ for $S'$ being a fat point. 

\sssec{}

To prove the existence, we will use a different construction of $\CL_{\Gr}$. Namely, we will use the following result,
proved below:

\begin{thm}  \label{t:raising to power}
For any $\kappa$, its multiple $N\cdot \kappa$ can be obtained from a (based) map 
$$\kappa':B_{\on{Zar}}(BG)\to \bK_{\geq 2}$$
by composing with the projection $\bK_{\geq 2}\to B^2_{\on{Zar}}(\bK_2)$.
\end{thm} 

\sssec{}

Let $\kappa'$ be as in \thmref{t:raising to power}. By \secref{ss:constr from full}, to such $\kappa'$ we can associate
a factorizable line bundle, to be denoted $\CL'_{\Gr}$ over $\Gr_{G,\Ran}$.   

\medskip

By \secref{sss:compat with global}, we have
\begin{equation} \label{e:comparing with power}
\CL'_{\Gr}\simeq \CL_{\Gr}^{\otimes N}
\end{equation} 
as plain line bundles.   

\medskip

The above isomorphism endows $\CL_{\Gr}^{\otimes N}$ with a factorization structure. By \secref{ss:Artin},
this factorization structure is compatible with the factorization structure on $\CL_{\Gr}$ evaluated on fat points.
By \lemref{l:full spectrum compat for smooth}, this factorization structure is compatible 
with the factorization structure on $\CL_{\Gr}$ evaluated on smooth schemes.

\medskip

Hence, it remains to show that the above factorization structure on $\CL_{\Gr}^{\otimes N}$ comes from a factorization structure on 
$\CL_{\Gr}$ in a way compatible with evaluation on smooth schemes. (The compatibility with the given factorization structure on $\CL_{\Gr}$
on fat points would follow, because, since $(\on{char}(k),N)=1$, if for an element $f\in \CO^\times_S$ we have $f^N=1$ and 
$f|_{S_{\on{red}}}=1$, then $f=1$.) 

\medskip

Now, the assertion follows from the next general result:

\begin{prop}  \label{p:Nth root}
For a line bundle $\CL'_Y$ on a scheme scheme $Y$, the datum of a line bundle
$\CL_Y$ equipped with an isomorphism $\CL_Y^{\otimes N}\simeq \CL'_Y$ is equivalent to the
datum of compatible family of line bundles $\CL_S$ equipped with isomorphisms 
$\CL_S^{\otimes N}\simeq \CL'_Y|_S$ for all \emph{smooth} affine schemes $S$ over $Y$.
\end{prop} 

\begin{proof}

The question is local with respect to the Zariski topology on $Y$, so we can assume that $\CL_Y$
is trivial. Then the question becomes that of comparing $\mu_N$-\'etale torsors on $Y$ with
compatible families of such on smooth affine schemes mapping to $Y$. 

\medskip

The result now follows from the fact that \'etale sheaves with coefficients that are torsion prime to
$\on{char}(k)$ satisfy descent with respect to the topology generated by proper surjective maps
and Zariski covers, and affine smooth schemes form form a basis for this topology.

\end{proof} 

\sssec{}

This completes the construction of $\CL_{\Gr}$ as a factorization line bundle. The rest of this section is devoted to the proof of
\thmref{t:raising to power}.

\ssec{Reduction to the case of $G=GL_n$}

\sssec{}

We will deduce the assertion of \thmref{t:raising to power} from its special case when $G=GL_n$. Namely, we will prove: 

\begin{thm} \label{t:lift for GL_n}
For $G=GL_n$, any map $B_{\on{Zar}}(BG)\to \bK_2$ can be lifted to a map
$B_{\on{Zar}}(BG)\to \bK_{\geq 2}$.
\end{thm} 

Let us show how \thmref{t:lift for GL_n} implies \thmref{t:raising to power}.  

\sssec{}

For a representation $V$ of $G$ consider the corresponding map
$$B_{\on{Zar}}(BG)\to B_{\on{Zar}}(BGL(V)),$$
and the induced map
\begin{equation} \label{e:from GLn}
\pi_0\left(\Maps(B_{\on{Zar}}(GL(V)),B^2_{\on{Zar}}(\bK_2))\right)\to \pi_0\left(\Maps(B_{\on{Zar}}(G),B^2_{\on{Zar}}(\bK_2))\right).
\end{equation} 

\medskip

Assuming \thmref{t:lift for GL_n}, in order to prove \thmref{t:raising to power}, it suffices to prove that the subgroup in 
$\pi_0\left(\Maps(B_{\on{Zar}}(G),B^2_{\on{Zar}}(\bK_2))\right)$, generated by the images of the maps \eqref{e:from GLn}
for all possible $V$, has a finite index and the quotient is annihilated by $N$. 

\sssec{}

Let us recall the definition of the integer $N$:

\medskip

Let $G_\BC$ be the complex Lie group corresponding to $G$.
We consider $G_\BC$ as a topological group, and consider its (topological) classifying space $BG_{\on{top}}$.
Consider the abelian group
$$H^4(BG_{\on{top}},\BZ).$$

\medskip

For each representation $V$ of $G$, consider the induced map
$$H^4(BGL(V)_{\on{top}},\BZ)\to H^4(BG_{\on{top}},\BZ).$$

The images of these maps generate a subgroup of finite index. We let $N$ be the smallest integer that annihilates 
the quotient. 
 
\sssec{}

We now prove the required assertion regarding $\pi_0\left(\Maps(B_{\on{Zar}}(G),B^2_{\on{Zar}}(\bK_2))\right)$. 

\medskip

Let $\Lambda_G$ denote the coweight lattice of $G$. According to \cite[Theorem 6.2]{BrDe}, we have an isomorphism of sets
$$\pi_0\left(\Maps(B_{\on{Zar}}(G),B^2_{\on{Zar}}(\bK_2))\right)\simeq \on{Quad}(\Lambda_G,\BZ)^{W_G},$$
where the latter is the set of Weyl group invariant
integer-valued quadratic forms on $\Lambda_G$. Moreover, this bijection respects the natural group structure on both sides,
and is functorial with respect to $G$. 

\medskip

Now, we have a canonical isomorphism
$$H^4(BG_{\on{top}},\BZ)\simeq \on{Quad}(\Lambda_G,\BZ)^{W_G},$$
and for $V$ defined over the integers, we have a commutative diagram
$$
\CD
H^4(BGL(V)_{\on{top}},\BZ)  @>>>   H^4(BG_{\on{top}},\BZ)  \\
@V{\sim}VV   @VV{\sim}V   \\
\on{Quad}(\Lambda_{GL(V)},\BZ)^{W_{GL(V)} }  @>>> \on{Quad}(\Lambda_G,\BZ)^{W_G} \\
@A{\sim}AA   @AA{\sim}A   \\
\Maps(B_{\on{Zar}}(GL(V)),B^2_{\on{Zar}}(\bK_2))   @>>> \Maps(B_{\on{Zar}}(G),B^2_{\on{Zar}}(\bK_2)).
\endCD
$$

This implies the required assertion. 

\ssec{Proof for $GL_n$}  \label{ss:GLn}

The goal of this section is to prove \thmref{t:lift for GL_n}. We set $G=GL_n$.  

\sssec{}  

The key observation that we will use is that for $G=GL_n$ the map
$$B_{\on{Zar}}(G)\to B_{\on{et}}(G)$$
is an isomorphism.

\medskip

In particular, if $G$ acts freely on a scheme $Z$ (i.e., the algebraic stack quotient $Y:=(Z/V)_{\on{et}}$ is actually a scheme),
the natural map
$$(Z/G)_{\on{Zar}}\to (Z/G)_{\on{et}}=Y$$
is an isomorphism, where $(Z/G)_{\on{Zar}}$ (resp., $(Z/G)_{\on{et}}$) denotes the sheafification of the prestack quotient $Z/G$ in the Zariski 
(resp., \'etale) topology.

\sssec{}   \label{sss:obs}

First, we note the obstruction to the existence of a lift belongs to
$$H^1(\tau^{\leq -3}(\CK_{\on{nc}})(B_{\on{Zar}}(G))).$$

Hence, to prove the theorem, it suffices to show the following:

\begin{prop}  \label{p:vanishing of higher cohomology}
The cohomology groups $H^j(\CK_i(B_{\on{Zar}}(G)))$ vanish for $j>i$.
\end{prop}

\begin{rem}
Note that the assertion of \propref{p:vanishing of higher cohomology} would follow from the Gersten resolution (see \secref{sss:Gersten})
if instead of $B_{\on{Zar}}(G)$ we had a smooth scheme. The proof will consist of showing that $B_{\on{Zar}}(G)$
behaves in a similar way. The argument below is inspired by \cite{To}. This idea was explained to us by A.~Beilinson. 
\end{rem}

\sssec{}  \label{sss:E and V}

According to \cite[Remark 1.4]{To}, for every integer $n$, we can find a representation $V$ of $G$ with the following property:
$V$ contains an open subset $E$, such that $\on{codim}(V-E,E)>n$, and such that the action of $G$ on $E$ is \emph{free}, i.e.,
the algebraic stack quotient 
$$(BG)^{\on{appr}}:=(E/G)_{\on{et}}$$
is a scheme. 

\medskip

We will show that
$$\CK_i(B_{\on{Zar}}(G))\simeq \CK_i((BG)^{\on{appr}})$$
for $i\leq n$. This would imply the assertion of the proposition.

\medskip

Consider the maps
\begin{equation} \label{e:seq of maps}
\CK_i(B_{\on{Zar}}(G))\to \CK_i((V/G)_{\on{Zar}})\to \CK_i((E/G)_{\on{Zar}})\simeq \CK_i((E/G)_{\on{et}})=
\CK_i((BG)^{\on{appr}}).
\end{equation}

We will show that the first of these maps is an isomorphism for all $i$, and the second is an isomorphism 
for $i\leq n$. (This will be true for any $G$, not just $G=GL_n$. It is the third isomorphism that uses the specifics
of $GL_n$.) 

\sssec{}

To prove that $$\CK_i((V/G)_{\on{Zar}})\to \CK_i((E/G)_{\on{Zar}})$$ is an isomorphism, we note that each side is computed as
the totalization of the corresponding co-simplicial complex:
$$j\rightsquigarrow \CK_i(G^{j-1}\times V) \text{ and }  j\rightsquigarrow  \CK_i(G^{j-1}\times E),$$
respectively. 

\medskip

However, the assumption that $\on{codim}(V-E,E)>i$ implies that the corresponding maps
$$\CK_i(G^{j-1}\times V) \to  \CK_i(G^{j-1}\times E)$$
are isomorphisms, by Gersten resolution. 

\sssec{}

To prove that 
$$\CK_i(B_{\on{Zar}}(G))\to \CK_i((V/G)_{\on{Zar}})$$
is an isomorphism, it suffices to show that the corresponding maps
$$\CK_i(G^{j-1}) \to \CK_i(G^{j-1}\times V)$$
are isomorphisms.

\medskip

However, this follows from the $\BA^1$-invariance of $\CK_i$: for a smooth scheme $S$, the map
$$\CK_i(S)\to \CK_i(S\times \BA^1)$$
is an isomorphism, according to \cite[Sect. 2]{Sh}. 

\begin{rem}  
Consider the triple $(V,E,(BG)^{\on{appr}})$ for an arbitrary connected reductive group, and the maps 
\eqref{e:seq of maps} for $i=2$. Let us take $H^2(-)$ of these complexes. 

\medskip

As is shown in \cite[Corollary 3.5]{To}, the resulting map
$$H^2(\CK_2(B_{\on{Zar}}(G)))\to H^2(\CK_2((BG)^{\on{appr}}))$$
is injective, but \emph{not} necessarily an isomorphism, even for
$G$ semi-simple simply connected. For example, it fails to be such for the group $\on{Spin}(7)$.
However, the fact that the isomorphism holds for $GL_n$ shows that the quotient is 
annihilated by the integer $N$. In fact, the image of this map equals the subgroup
generated by the images of the maps \eqref{e:from GLn} for all $V$, i.e., the subgroup of 
$$\on{Quad}(\Lambda_G,\BZ)^{W_G}\simeq H^4(BG_{\on{top}},\BZ)$$
generated by Chern classes of representations.
\end{rem} 

\section{Construction in the general case}  \label{s:construction gen}

\ssec{Recollections from \cite{BrDe}}  \label{ss:BrDe}

\sssec{}

Let us recall the theorem from \cite{BrDe} (namely, Theorem 6.2) that describes explicitly the Picard category 
$$\Maps_{\on{based}}(Y\times B_{\on{Zar}}(G),B^2_{\on{Zar}}(\bK_2))$$
for any smooth and connected $Y$. 

\medskip

This theorem says that this Picard category is canonically equivalent to one consisting of the following data:

\begin{itemize}

\item A Weyl group invariant quadratic form $Q$ on the coweight $\Lambda_G$ lattice of $G$;

\item A central extension $\CE$ of $\Lambda_G$ (viewed as a constant Zariski sheaf on $Y$) by $\CO^\times_Y$,
such that the commutator pairing $\Lambda_Q\otimes \Lambda_Q\to \Gamma(X,\CO^\times_X)$ is given by
$$\lambda_1,\lambda_2\mapsto (-1)^{B(\lambda_1,\lambda_2)},$$
where $B$ is the bilinear form corresponding to $Q$;

\item An identification of the pullback of $\CE$ to $\Lambda_{G_{\on{sc}}}$ with a canonical one $\CE_{\on{sc}}$ arising 
from $Q|_{\Lambda_{G_{\on{sc}}}}$ (here $\Lambda_{G_{\on{sc}}}\subset \Lambda_G$ is the coroot lattice of $G$=
coweight lattice of the simply connected cover $G_{\on{sc}}$ of the derived group of $G$).

\end{itemize}

\medskip

We will not need the full force of this theorem (in particular, we will not need to know what $\CE_{\on{sc}}$ is). 
Rather, we will use some of its consequences. 

\sssec{}

The description of $\Maps_{\on{based}}(Y\times B_{\on{Zar}}(G),B^2_{\on{Zar}}(\bK_2))$ given above implies that there
exists a natural map
$$\Maps_{\on{based}}(Y\times B_{\on{Zar}}(G),B^2_{\on{Zar}}(\bK_2))\to \on{Quad}(\Lambda_G,\BZ)^{W_G},$$
which is surjective on $\pi_0$ (because this is so for $Y=\on{pt}$, in which case this map is an isomorphism on
 $\pi_0$).  
 
 \medskip
 
Let us denote by $\Maps_{\on{based}}(Y\times B_{\on{Zar}}(G),B^2_{\on{Zar}}(\bK_2))^0$ its fiber.
Here is the input from the \cite{BrDe} description that we will use:

\begin{cor} \label{c:BD 1}
The Picard category $\Maps_{\on{based}}(Y\times B_{\on{Zar}}(G),B^2_{\on{Zar}}(\bK_2))^0$ is canonically
equivalent to the Picard category 
$$\on{Fib}\left(\on{Tors}_{\on{Zar}}(\Hom(\Lambda_G,\CO^\times_Y))\to \on{Tors}_{\on{Zar}}(\Hom(\Lambda_{G_{\on{sc}}},\CO^\times_Y))\right).$$
\end{cor}

Note now that we have a canonical equivalence of Picard categories 
\begin{multline*}
\on{Fib}\left(\on{Tors}_{\on{Zar}}(\Hom(\Lambda_G,\CO^\times_Y))\to \on{Tors}_{\on{Zar}}(\Hom(\Lambda_{G_{\on{sc}}},\CO^\times_Y))\right) \simeq \\
\simeq \on{Fib}\left(\on{Tors}_{\on{et}}(\Hom(\Lambda_G,\CO^\times_Y))\to \on{Tors}_{\on{et}}(\Hom(\Lambda_{G_{\on{sc}}},\CO^\times_Y))\right)\simeq  \\
\simeq \on{Tors}_{\on{et}}(\Hom(\pi_{1,\on{alg}}(G),\CO^\times_Y)),
\end{multline*}
where $\pi_{1,\on{alg}}(G):=\Lambda_G/\Lambda_{G_{\on{sc}}}$ denotes the algebraic fundamental group of 
$G$. 

\sssec{}

As a consequence, we obtain: 

\begin{cor}  \label{c:BD 2}
There exists a canonical homomorphism of Picard categories
$$\on{Tors}_{\on{et}}(\Hom(\pi_{1,\on{alg}}(G),\CO^\times_Y))\to \Maps_{\on{based}}(Y\times B_{\on{Zar}}(G),B^2_{\on{Zar}}(\bK_2)),$$
and the induced map
\begin{multline*}
\on{Tors}_{\on{et}}(\Hom(\pi_{1,\on{alg}}(G),\CO^\times_Y)) \overset{\on{Tors}(\Hom(\pi_{1,\on{alg}}(G),k^\times))}\times 
\Maps_{\on{based}}(B_{\on{Zar}}(G),B^2_{\on{Zar}}(\bK_2))\to \\
\to \Maps_{\on{based}}(Y\times B_{\on{Zar}}(G),B^2_{\on{Zar}}(\bK_2))
\end{multline*} 
is an equivalence.
\end{cor} 

From this corollary, we obtain:

\begin{cor}  \label{c:BD 3}
\begin{multline*}
\pi_i(\Maps_{\on{based}}(Y\times B_{\on{Zar}}(G),B^2_{\on{Zar}}(\bK_2)))\simeq \\
\simeq \begin{cases}
&\on{Quad}(\Lambda_G,\BZ)^{W_G}\times H^1_{\on{et}}(Y,\Hom(\pi_{1,\on{alg}}(G),\CO^\times_Y)) \text{ for } i=0; \\
&H^0(Y,\Hom(\pi_{1,\on{alg}}(G),\CO^\times_Y)) \text{ for } i=1.
\end{cases}
\end{multline*} 
\end{cor} 

\ssec{Reduction to the case of tori}  

\sssec{}  \label{sss:steps general}

Our goal is to construct a homomorphism of Picard groupoids:
\begin{equation}  \label{e:functor for non-complete}
\Maps_{\on{based}}(X\times B_{\on{Zar}}(G),B^2_{\on{Zar}}(\bK_2))\to \on{FactLine}(\Gr_{G,\Ran}(X)). 
\end{equation}

Given \corref{c:BD 2} and the already constructed map 
$$\Maps_{\on{based}}(B_{\on{Zar}}(G),B^2_{\on{Zar}}(\bK_2))\to \on{FactLine}(\Gr_{G,\Ran}(X))$$
(see \secref{s:construction abs}), we need to perform the following two steps:

\begin{itemize}

\item Construct a homomorphism of Picard categories
$$\on{Tors}_{\on{et}}(\Hom(\pi_{1,\on{alg}}(G),\CO_X^\times))\to \on{FactLine}(\Gr_{G,\Ran}(X));$$

\item Establish the commutativity of the diagram 
\begin{equation} \label{e:trivial diagram}
\CD
\on{Tors}(\Hom(\pi_{1,\on{alg}}(G),k^\times))   @>>>   \on{Tors}_{\on{et}}(\Hom(\pi_{1,\on{alg}}(G),\CO_X^\times)) \\
@VVV   @VVV  \\
\Maps_{\on{based}}(B_{\on{Zar}}(G),B^2_{\on{Zar}}(\bK_2))  @>>> \on{FactLine}(\Gr_{G,\Ran}(X))
\endCD
\end{equation} 

\end{itemize}

\sssec{}

We will carry out these two steps using an additional choice. Namely, we will choose a short exact sequence
\begin{equation} \label{e:cover G}
1\to T \to \wt{G}\to G\to 1,
\end{equation}
where $T$ is a torus, and $\wt{G}$ is a reductive group whose derived group $\wt{G}'$ is simply connected.

\medskip

The fact that the result is independent of this choice will be evident from the construction. 

\sssec{}  

Note that the torus $T$, being commutative, can be regarded as a group-object in the category of groups. Hence,
$\Gr_{T,\Ran}$ acquires a structure a relative group ind-scheme over $\Ran$. 

\medskip

We have an action of $T$, viewed as a group-object in he category of groups, on $\wt{G}$,
and the action map
$$T\times \wt{G}\to \wt{G}\underset{G}\times \wt{G}$$
is an isomorphism. 

\medskip

Hence, we obtain an action the relative group ind-scheme $\Gr_{T,\Ran}$ over $\Ran$ on $\Gr_{\wt{G},\Ran}$ so that the resulting map
\begin{equation} \label{e:fiber Gr}
\Gr_{T\times \wt{G},\Ran}\simeq
\Gr_{T,\Ran}\underset{\Ran}\times \Gr_{\wt{G},\Ran}\to \Gr_{\wt{G},\Ran}\underset{\Gr_{G,\Ran}}\times \Gr_{\wt{G},\Ran}
\end{equation}
is an isomorphism. 

\medskip

Note also that the map 
$$\Gr_{\wt{G},\Ran}\to \Gr_{G,\Ran}$$
is surjective in the topology where we take as coverings finite surjective maps.

\begin{rem}
One can show that the above map is surjective in the \'etale topology, but the proof that we have in mind 
is quite involved, so we will avoid using this fact.
\end{rem}

\sssec{}  \label{sss:derived}

Consider the topology of finite surjective maps in the setting of \emph{derived schemes}. We claim:

\begin{prop} \label{p:fin surj}
Let $f:\wt{S}\to S$ be a map of derived schemes (assumed locally almost of finite type), 
such that the resulting map of classical schemes $\wt{S}_{\on{cl}}\to S_{\on{cl}}$ is finite an surjective. 
Then the groupoid of line bundles on $S$ maps isomorphically to the totalization of the
cosimplicial groupoid of line bundles on the \v{C}ech nerve of $f$. 
\end{prop}

\begin{proof}

For any derived affine scheme $S$ (assumed locally almost of finite type) the functor of tensoring by
the dualizing object
$$\CF\mapsto \CF\otimes \omega_S$$
defines a \emph{fully faithful} functor
$$\on{Perf}(S)\to \IndCoh(S).$$

Hence, it suffices to prove the corresponding descent statement for $\IndCoh$. However, in the latter
case, it is given by \cite[Proposition 8.2.3]{Ga}. 

\end{proof} 

\begin{cor}
There exists an equivalence of categories between $\on{FactLine}(\Gr_{G,\Ran})$ and the category
of objects in $\on{FactLine}(\Gr_{\wt{G},\Ran})$, factorizably equivariant with respect $\Gr_{T,\Ran}$. 
\end{cor}

\begin{rem}
Here we are using the fact that the isomorphism \eqref{e:fiber Gr} holds at the \emph{derived level},
see Remark \ref{r:Gr class}.
\end{rem}

\sssec{}

Hence, if we can perform the two steps in \secref{sss:steps general} for groups of the form 
$\wt{G}\times T^n$, in a way compatible with the maps in the simplicial group
that encodes the action of $T$ on $\wt{G}$, this would imply the two steps in \secref{sss:steps general} for $G$. 

\sssec{}

Let $\wt{T}$ denote the quotient of $\wt{G}$ by its derived group. Note that the map
$$\pi_{1,\on{alg}}(\wt{G})\to \pi_{1,\on{alg}}(\wt{T})$$
is an isomorphism. 

\medskip

Hence, in order to carry out the two steps in \secref{sss:steps general} for groups of the form $\wt{G}\times T^n$ 
(in a way compatible with the simplicial structure), it suffices to so for $\wt{G}$ replaced by $\wt{T}$. 

\medskip

The compatibility with the simplicial structure will be automatically encoded by the functoriality of the construction 
with respect to maps of tori.  

\ssec{Construction for tori}

\sssec{}

Recall that $\on{Tors}_{\on{et}}(\Hom(\pi_{1,\on{alg}}(T),\CO_Y^\times))$ is a Picard sub-groupoid of 
$$\Maps_{\on{based}}(X\times B_{\on{Zar}}(T),B^2_{\on{Zar}}(\bK_2)).$$ Hence, in order to perform
the two steps of \secref{sss:steps general}, it would suffice to construct the initial map
\begin{equation} \label{e:constr for tori}
\Maps_{\on{based}}(X\times B_{\on{Zar}}(T),B^2_{\on{Zar}}(\bK_2))\to  \on{FactLine}(\Gr_{T,\Ran})
\end{equation} 
in the case when $G=T$ is a torus. 

\medskip

We will perform this construction in the present subsection. 

\sssec{}

Let $\Lambda\simeq \pi_{1,\on{alg}}(T)$ denote the lattice of cocharacters of $T$. We define the prestack
$\Gr_{\Lambda,\Ran}$ as follows. 

\medskip

For an affine test-scheme $S$, its $S$-points are pairs $(I,\lambda_I)$,
where $I\subset \Hom(S,X)$ is an $S$-point of $\Ran(X)$, and $\lambda_I$ is a map $I\to \Lambda$. 

\medskip

For a map $f:S'\to S$ and $I'\subset \Hom(S',X)$ being the image of $i$, we let
$$\lambda_{I'}:I'\to \Lambda$$ be defined by
$$\lambda_{I'}(i')=\underset{i\in I,\, i\mapsto i'}\Sigma\, \lambda_I(i).$$

\sssec{}

We have a natural projection $$\Gr_{\Lambda,\Ran}\to \Ran,$$ so we can talk about the Picard category
$$\on{FactLine}(\Gr_{\Lambda,\Ran}(X))$$ of factorization line bundles on $\Gr_{\Lambda,\Ran}$.

\sssec{}

We have a naturally defined map 
\begin{equation} \label{e:Gr T comb}
\Gr_{\Lambda,\Ran}\to \Gr_{T,\Ran}
\end{equation} 
that sends $(I,\lambda_I)$ to the $T$-bundle on $S\times X$ given by
$$\underset{i]in I}\bigotimes\, \lambda_I(i)\cdot \CO(\on{Graph}_{x_i}),$$
where we denote by $x_i$ the map $S\to X$ corresponding to $i\in I$, so that its 
graph $\on{Graph}_{x_i}$ is a Cartier divisor in $S\times X$. 

\medskip

It is easy to see that the map \eqref{e:Gr T comb} is an isomorphism after sheafification
with respect to the topology generated by finite surjective maps. Hence, by \propref{p:fin surj}, 
the pullback map
$$\on{FactLine}(\Gr_{T,\Ran})\to \on{FactLine}(\Gr_{\Lambda,\Ran})$$
is an equivalence.

\sssec{}

Hence, in order to construct \eqref{e:constr for tori}, it suffices to construct the corresponding map
$$\Maps_{\on{based}}(X\times B_{\on{Zar}}(T),B^2_{\on{Zar}}(\bK_2))\to \on{FactLine}(\Gr_{\Lambda,\Ran}).$$

However, the latter follows from \secref{ss:smooth}: indeed, the prestack $\Gr_{\Lambda,\Ran}$ is a colimit
of smooth schemes.

\section{Some conjectures}  \label{s:mot}

\ssec{Conjecture about equivalence}

\sssec{}

In the preceding sections we constructed a map of Picard groupoids:
\begin{equation} \label{e:the map}
\Maps_{\on{based}}(X\times B_{\on{Zar}}(G),B^2_{\on{Zar}}(\bK_2))\to \on{FactLine}(\Gr_{G,\Ran}).
\end{equation} 

We propose:

\begin{conj} \label{c:main}
The map \eqref{e:the map} is an isomorphism of Picard groupoids.
\end{conj}

\sssec{}

Assume for a moment that $G=T$ is a torus.  Then one can show that the Picard category $\on{FactLine}(\Gr_{T,\Ran})$
identifies with the category of \emph{even} $\theta$-data of \cite[Sect. 3.10.3]{BD}. In particular,
$$
\pi_i(\on{FactLine}(\Gr_{T,\Ran}))=
\begin{cases}
&\on{Quad}(\Lambda,\BZ)\times H^1(X,\Hom(\Lambda,\CO^\times_X)) \text{ for } i=0;\\
&\Gamma(X,\Hom(\Lambda,\CO^\times_X)) \text{ for } i=1.
\end{cases}
$$

\medskip

According to \corref{c:BD 3}, the Picard groupoid $\Maps_{\on{based}}(X\times B_{\on{Zar}}(T),B^2_{\on{Zar}}(\bK_2))$
has homotopy groups given by the same expression. 

\medskip

It should not be difficult to see by unwinding the constructions, that in terms of these identifications, 
the map \eqref{e:the map} induces the identity maps on $\on{Quad}(\Lambda,\BZ)\times H^1(X,\Hom(\Lambda,\CO^\times_X))$ and 
$\Gamma(X,\Hom(\Lambda,\CO^\times_X))$, 
implying that \conjref{c:main} holds for tori.

\ssec{Factorizable line bundes vs factorizable gerbes}

\sssec{}

Let $\ell$ be an integer prime to $\on{char}(k)$. 
Following \cite{GL}, we consider the category (in fact, a Picard 2-category) $\on{FactGerbe}_{\mu_\ell}(\Gr_{G,\Ran})$
of factorizable \'etale $\mu_\ell$-gerbes on $\Gr_{G,\Ran}$.  According to \cite[Proposition 3.2.2]{GL}, we have a canonical equivalance
\begin{equation} \label{e:param gerbe}
\Maps_{\on{based}}(X\times B_{\on{Zar}}(G),B^4_{\on{et}}(\mu^{\otimes 2}_\ell))\simeq \on{FactGerbe}_{\mu_\ell}(\Gr_{G,\Ran}),
\end{equation} 
where we also note that 
$$\Maps_{\on{based}}(X\times B_{\on{Zar}}(G),B^4_{\on{et}}(\mu^{\otimes 2}_\ell))\simeq
\Maps_{\on{based}}(X\times B_{\on{et}}(G),B^4_{\on{et}}(\mu^{\otimes 2}_\ell)).$$

\sssec{}

The Kummer sequence of \'etale sheaves (for any scheme $Y$)
$$0\to \mu_\ell\to \CO_Y^\times \overset{f\mapsto f^\ell}\longrightarrow \CO_Y^\times\to 0$$
defines a map (in fact, a monomorphism) of Picard 2-categories
$$\{\text{Line bundles on }Y\}/\ell\simeq \{\mu_\ell\text{-gerbes on }Y\}.$$

\medskip

In particular, we obtain a map (in fact, a monomorphism) 
\begin{equation} \label{e:Kummer on Gr mod l}
\on{FactLine}(\Gr_{G,\Ran})/\ell \to \on{FactGerbe}_{\mu_\ell}(\Gr_{G,\Ran}).
\end{equation}

\sssec{}

Hence, we obtain a canonical map 
\begin{equation} \label{e:from K to mu mod l}
\Maps_{\on{based}}(X\times B_{\on{Zar}}(G),B^2_{\on{Zar}}(\bK_2))/\ell\to \Maps_{\on{based}}(X\times B_{\on{Zar}}(G),B^4_{\on{et}}(\mu^{\otimes 2}_\ell)),
\end{equation}
that makes the diagram
\begin{equation} \label{e:param diag}
\CD
\Maps(X\times B_{\on{Zar}}(G),B^2_{\on{Zar}}(\bK_2)))/\ell    @>{\text{\eqref{e:from K to mu mod l}}}>>  
\Maps(X\times B_{\on{Zar}}(G),B^4_{\on{et}}(\mu^{\otimes 2}_\ell)) \\
@V{\text{\eqref{e:the map}}/\ell}VV   @V{\text{\eqref{e:param gerbe}}}V{\sim}V   \\
\on{FactorLine}(\Gr_{G,\Ran})/\ell  @>{\text{\eqref{e:Kummer on Gr mod l}}}>>  \on{FactGerbe}_{\mu_\ell}(\Gr_{G,\Ran})
\endCD
\end{equation}
commute. 

\begin{rem}
Note that above we constructed the map \eqref{e:from K to mu mod l} by appealing to factorizable line bundles/gerbes,
rather than ``abstractly", i.e., by mapping various target prestacks to one-another.  In fact, we do not know how to construct
such a map directly: 
for example, it is easy to see that there \emph{does not} exist a non-trivial map
$$\bK_2/\ell \to B^2_{\on{et}}(\mu_\ell^{\otimes 2}).$$
(indeed, the \'etale sheafification of $\bK_2/\ell$ is zero, by Suslin's rigidity). 

\medskip

However, in the next subsection, we will give a construction of a certain map (that conjecturally equals)
\eqref{e:from K to mu mod l}, using motivic cohomology.

\end{rem}

\sssec{}

Recall (see \cite[Theorem 3.2.6]{GL}) that the \'etale cohomology of $B_{\on{Zar}}(G)$ is given by:
$$
H^i(B_{\on{Zar}}(G),\BZ_\ell(2))= 
\begin{cases}
&\on{Quad}(\Lambda_G,\BZ_\ell)^W \text{ for } i=4; \\
&\Ext^1(\pi_{1,\on{alg}}(G),\BZ_\ell)(1) \text{ for } i=3; \\
&\Hom(\pi_{1,\on{alg}}(G),\BZ_\ell)(1) \text{ for } i=2; \\
&0 \text{ for } i=1.
\end{cases}
$$

From here, by K\"unneth formula, we obtain:  
\begin{multline*} 
H^i(B_{\on{Zar}}(G)\times X,\BZ_\ell(2))= \\
\begin{cases}
&\on{Quad}(\Lambda_G,\BZ_\ell)^W \times H^2(X,\Hom(\pi_{1,\on{alg}}(G),\BZ_\ell)(1)) \times
H^1(X,\Ext^1(\pi_{1,\on{alg}}(G),\BZ_\ell)(1)) \text{ for } i=4;\\
&H^1(X,\Hom(\pi_{1,\on{alg}}(G),\BZ_\ell)(1)) \times \Ext^1(\pi_{1,\on{alg}}(G),\BZ_\ell)(1) \text{ for } i=3; \\
&\Hom(\pi_{1,\on{alg}}(G),\BZ_\ell)(1)\text{ for } i=2; \\
&0 \text{ for } i=1.
\end{cases}
\end{multline*} 

Hence, the homotopy groups of 
$$\pi_i(\Maps_{\on{based}}(X\times B_{\on{Zar}}(G),B^4_{\on{et}}(\BZ_\ell(2)))/\ell)$$
are given by

\begin{equation} \label{e:mod ell homotopy}
\begin{cases}
&\left(\on{Quad}(\Lambda_G,\BZ)^W\underset{\BZ}\otimes \BZ/\ell\right) \times H^2(X,\Hom(\pi_{1,\on{alg}}(G),\mu_\ell)) \times
H^1(X,\Ext^1(\pi_{1,\on{alg}}(G),\mu_\ell)) \text{ for } i=0;\\
&H^1(X,\Hom(\pi_{1,\on{alg}}(G),\mu_\ell)) \times \Ext^1(\pi_{1,\on{alg}}(G),\mu_\ell) \text{ for } i=1; \\
&\Hom(\pi_{1,\on{alg}}(G),\mu_\ell)\text{ for } i=2; \\
\end{cases}
\end{equation}

\sssec{}

Note now that we have a fully faithful map of Picard 2-groupoids
\begin{equation} \label{e:l business}
\Maps_{\on{based}}(X\times B_{\on{Zar}}(G),B^4_{\on{et}}(\BZ_\ell(2)))/\ell\to 
\Maps_{\on{based}}(X\times B_{\on{Zar}}(G),B^4_{\on{et}}(\mu^{\otimes 2}_\ell)).
\end{equation}

Its essential image corresponds to the subgroup of
$$H^4_{\on{et}}(X\times B_{\on{Zar}}(G);X\times \on{pt},\mu^{\otimes 2}_\ell)$$
equal to the image of the map 
\begin{equation} \label{e:l business bis}
H^4_{\on{et}}(X\times B_{\on{Zar}}(G);X\times \on{pt},\BZ_\ell(2))\to 
H^4_{\on{et}}(X\times B_{\on{Zar}}(G);X\times \on{pt},\mu^{\otimes 2}_\ell),
\end{equation}
which is also the kernel of the map 
$$H^4_{\on{et}}(X\times B_{\on{Zar}}(G);X\times \on{pt},\mu^{\otimes 2}_\ell) \to 
H^5_{\on{et}}(X\times B_{\on{Zar}}(G);X\times \on{pt},\BZ_\ell(2))$$
in the long exact cohomology sequence associated with the short exact sequence of \'etale sheaves
$$0\to \BZ_\ell(2) \overset{\ell}\to \BZ_\ell(2) \to \mu^{\otimes 2}_\ell\to 0.$$

\sssec{}

Observe that that \corref{c:BD 3} implies that the homotopy groups of $$\Maps_{\on{based}}(X\times B_{\on{Zar}}(G),B^2_{\on{Zar}}(\bK_2))/\ell$$
are also canonically identified with \eqref{e:mod ell homotopy}. In light of this, we propose:

\begin{conj} \label{c:compat param}
The map \eqref{e:from K to mu mod l} factors as
\begin{multline*} 
\Maps_{\on{based}}(X\times B_{\on{Zar}}(T),B^2_{\on{Zar}}(\bK_2))/\ell\to \Maps_{\on{based}}(X\times B_{\on{Zar}}(G),B^4_{\on{et}}(\BZ_\ell(2)))/\ell \hookrightarrow \\
\hookrightarrow \Maps_{\on{based}}(X\times B_{\on{Zar}}(G),B^4_{\on{et}}(\mu^{\otimes 2}_\ell)),
\end{multline*} 
and the first arrow is an equivalence. 
\end{conj}

\begin{rem}

Recall from \cite[Theorem 3.2.6]{GL} that the cohomology $H^i(B_{\on{Zar}}(G),\mu_\ell^{\otimes 2})$ is given by
$$
\begin{cases}
&\on{Quad}(\Lambda_G,\BZ/\ell)^W_{\on{restr}} \text{ for } i=4; \\
&\Ext^1(\pi_{1,\on{alg}}(G),\mu_\ell) \text{ for } i=3; \\
&\Hom(\pi_{1,\on{alg}}(G),\mu_\ell) \text{ for } i=2; \\
&0 \text{ for } i=1,
\end{cases}
$$
where 
$$\on{Quad}(\Lambda_G,\BZ/\ell)^W_{\on{restr}}\subset \on{Quad}(\Lambda_G,\BZ/\ell)^W$$
is the subgroup defined in \cite[Sect. 3.2.2]{GL}.

\medskip

Hence, by K\"unneth formula, we obtain:
\begin{multline*} 
H^i(B_{\on{Zar}}(G)\times X,\mu_\ell^{\otimes 2})= \\
\begin{cases}
&\on{Quad}(\Lambda_G,\BZ/\ell)^W_{\on{restr}}\times H^2(X,\Hom(\pi_{1,\on{alg}}(G),\mu_\ell)) \times
H^1(X,\Ext^1(\pi_{1,\on{alg}}(G),\mu_\ell)) \text{ for } i=4;\\
&H^1(X,\Hom(\pi_{1,\on{alg}}(G),\mu_\ell)) \times \Ext^1(\pi_{1,\on{alg}}(G),\mu_\ell) \text{ for } i=3; \\
&\Hom(\pi_{1,\on{alg}}(G),\mu_\ell)\text{ for } i=2; \\
\end{cases}
\end{multline*} 

Hence, the image of the map \eqref{e:l business bis} corresponds to the subgroup
$$\on{Quad}(\Lambda_G,\BZ)^W\underset{\BZ}\otimes \BZ/\ell\subset \on{Quad}(\Lambda_G,\BZ/\ell)^W_{\on{restr}}.$$

This inclusion is an equality in many cases; in particular whenever the derived group of $G$ is simply-connected. 

\end{rem} 
\ssec{Motivic cohomology vs K-theory}  \label{ss:mot}

\sssec{}

For a smooth scheme $Y$, let $\BZ_{\on{mot}}(n)\in \on{Shv}_{\on{Zar}}(Y,\on{Ab})$ be Voevodsky's motivic complex,
see \cite[Lecture 3]{MVW}. It is known to live in the cohomological degrees $\leq n$. 

\sssec{}

Define the complex
$$\Gamma(X\times B_{\on{Zar}}(G),\BZ_{\on{mot}}(n))$$
as the totalization of 
$$\Gamma(X\times B^\bullet(G),\BZ_{\on{mot}}(n)),$$
where $B^\bullet(G)$ is the standard simplicial model of $B(G)$
$$...G\rightrightarrows \on{pt}.$$

\sssec{} \label{sss:mot to K}

Recall that for any smooth scheme $Y$, we have a canonical map in $\Shv_{\on{Zar}}(Y,\on{Ab})$
\begin{equation} \label{e:from mot to K}
\BZ_{\on{mot}}(2)[2]\to \CK_2|_Y.
\end{equation}
This follows, e.g., from \cite[Theorem 5.1]{MVW}, combined
with the Gersten resolution (see \cite[Corollary 6.3.3]{BV}). 

\medskip

By construction, the map
$$\on{Fib}(\BZ_{\on{mot}}(2)[2])\to \CK_2|_Y) \to (j_Y)_*\circ (j_Y)^*\left(\on{Fib}(\BZ_{\on{mot}}(2)[2])\to \CK_2|_Y)\right),$$
is an isomorphism, where $\eta_Y$ is the embedding of the generic point of $Y$, where we note that 
$(j_Y)^*\left(\on{Fib}(\BZ_{\on{mot}}(2)[2]\to \CK_2|_Y)\right)$ is a complex living in \emph{negative} cohomological degrees. 

\medskip

From here, it follows that \eqref{e:from mot to K} induces an isomorphism the $\tau^{\geq 0}$ truncations, and so does the induced map 
\begin{equation} \label{e:from mot to K Gamma}
\Gamma(Y,\BZ_{\on{mot}}(2)[2])\to \CK_2(Y). 
\end{equation}

\medskip

Note also that the complexes in \eqref{e:from mot to K Gamma} have cohomologies in degrees $\leq 2$, with $H^2(-)$ being $\on{CH}_2(Y)$. 

\sssec{}

From \eqref{e:from mot to K Gamma}  we obtain a map
\begin{equation} \label{e:from mot to K BG}
\Gamma(X\times B_{\on{Zar}}(G),\BZ_{\on{mot}}(2)[2])\to  
\CK_2(X\times B_{\on{Zar}}(G)). 
\end{equation}
(Note that it is no longer clear that the two sides in \eqref{e:from mot to K BG} have cohomologies in degrees $\leq 2$.)

\medskip

In the next theorem, $X$ can be an arbitrary smooth scheme over $k$. 

\begin{thm} \label{t:mot}
The map \eqref{e:from mot to K BG} defines an equivalence of the $\tau^{\geq 0}$ truncations. 
\end{thm}

This theorem is contained in \cite{EKLV}, but may not be stated there explicitly. 
We will supply a proof, reproducing the arguments from {\it loc.cit.} 

\sssec{}

Let us assume \thmref{t:mot} and show how it allows to construct a map
\begin{equation} \label{e:from K to mu mod l again}
\Maps_{\on{based}}(X\times B_{\on{Zar}}(G),B^2_{\on{Zar}}(\bK_2))/\ell\to \Maps_{\on{based}}(X\times B_{\on{Zar}}(G),B^4_{\on{et}}(\mu^{\otimes 2}_\ell)),
\end{equation}
or, equivalently, a map
\begin{equation} \label{e:from K to mu}
\Maps_{\on{based}}(X\times B_{\on{Zar}}(G),B^2_{\on{Zar}}(\bK_2))\to \Maps_{\on{based}}(X\times B_{\on{Zar}}(G),B^4_{\on{et}}(\mu^{\otimes 2}_\ell)),
\end{equation}

Indeed, from \thmref{t:mot} we obtain that
$$\Maps_{\on{based}}(X\times B_{\on{Zar}}(G),B^2_{\on{Zar}}(\bK_2))$$ identifies with 
$$\tau^{\geq -2,\leq 0}\biggl(\Gamma(X\times B_{\on{Zar}}(G);X\times \on{pt},\BZ_{\on{mot}}(2)[4])\biggr),$$
while $\Maps_{\on{based}}(X\times B_{\on{Zar}}(G),B^4_{\on{et}}(\mu^{\otimes 2}_\ell))$ identifies with
$$\tau^{\leq 0}\biggl(\Gamma_{\on{et}}(X\times B_{\on{Zar}}(G);X\times \on{pt},\mu_\ell^{\otimes 2}[4])\biggr),$$
which projects isomorphically to
$$\tau^{\geq -2,\leq 0}\biggl(\Gamma_{\on{et}}(X\times B_{\on{Zar}}(G);X\times \on{pt},\mu_\ell^{\otimes 2}[4])\biggr).$$

\medskip

Now, the required map follows from the canonical assingment for any smooth scheme $S$ of a map
\begin{equation} \label{e:from Z to mu}
\Gamma(S,\BZ_{\on{mot}}(n))\to \Gamma(S,\BZ/\ell_{\on{mot}}(n))\to \Gamma_{\on{et}}(S,\mu_\ell^{\otimes n}).
\end{equation}

Indeed, the \'etale sheafification of $\BZ_{\on{mot}}(n)/\ell$ is isomorphic to $\mu_\ell^{\otimes n}$
(this is \cite[Theorem 10.2]{MVW}). 

\sssec{}

We propose:

\begin{conj} \label{c:param compat}
The diagram of equivalences
$$
\CD
\Maps(X\times B_{\on{Zar}}(T),B^2_{\on{Zar}}(\bK_2)))/\ell    @>{\text{\eqref{e:from K to mu mod l again}}}>>  
\Maps(X\times B_{\on{Zar}}(G),B^4_{\on{et}}(\mu^{\otimes 2}_\ell)) \\
@V{\text{\eqref{e:the map}}/\ell}VV   @VV{\text{\eqref{e:param gerbe}}}V   \\
\on{FactorLine}(\Gr_{G,\Ran})/\ell  @>{\text{\eqref{e:Kummer on Gr mod l}}}>>  \on{FactGerbe}_{\mu_\ell}(\Gr_{G,\Ran})
\endCD
$$
commutes.
\end{conj}

\sssec{Proof of \thmref{t:mot}}  %\label{ss:plaus Z2}

We essentially reproduce the proof from \cite[Sect. 6]{EKLV}. The assertion of the theorem 
is valid not just for $X\times B_{\on{Zar}}(G)$ but for prestacks $\CY$ of the following form:

\medskip

We need $\CY$ to be the geometric realization (=colimit) of a smooth simplicial prestack $Y^\bullet$ 
(either sheafified in Zariski topology or not) over a smooth scheme $X$, such that each term 
$Y^j$ is a \emph{rational} variety over the generic point of $X$ (see \cite[Lemma 6.2]{EKLV}
for what this means). 

\medskip

Consider again the map \eqref{e:from mot to K Gamma}. By \secref{sss:mot to K}, its fiber has the form 
$$(j_Y)_*(L_{\eta_Y}),$$
where $\eta_Y\overset{j_Y}\hookrightarrow Y$ is the embedding of the generic point, and $L_{\eta_Y}$ is a complex
of abelian groups living in degrees $\leq -1$. 

\medskip

Hence, it suffices to show that the totalization of the cosimplicial complex that attaches to $j$
\begin{equation} \label{e:cosimp gen}
L_{\eta_{Y^j}}
\end{equation} 
is acyclic in degrees $\geq 0$.

\medskip

The crucial observation is that the above cosimplicial complex is constant with value $L_{\eta_X}$. This would
imply that its totalization, which is isomorphic to $L_{\eta_X}$, lives in degrees $\leq -1$.

\medskip

To prove that \eqref{e:cosimp gen} is constant, it is enough to show that if $Y\to X$ is a morphism 
of smooth varieties such that $Y$ is rational over the generic point of $X$, then the map
$$L_{\eta_X}\to L_{\eta_Y}$$
is an isomorphism. As the assertion only depends on the generic point if $Y$, we can assume that
$Y$ is of the form $X\times \BA^n$. 

\medskip

We have a map of fiber sequences
$$
\CD
L_{\eta_{Y}}   @>>>  \Gamma(Y,\BZ_{\on{mot}}(2)[2])  @>>> \Gamma(Y,\CK_2|_Y) \\
@AAA  @AAA  @AAA \\ 
L_{\eta_{X}}   @>>>  \Gamma(X,\BZ_{\on{mot}}(2)[2])  @>>> \Gamma(X,\CK_2|_X). 
\endCD
$$

Now, the right and middle vertical arrows are isomorphisms by the $\BA^1$-invariance of $\Gamma(-,\BZ_{\on{mot}}(n))$
and $\Gamma(-,\CK_2)$, respectively. Hence, the left vertical arrow is an isomorphism, as required. 

\ssec{Comparing with \'etale motivic cohomology}

\sssec{}

Consider now the \emph{\'etale} motivic cohomology $\Gamma_{\on{et}}(X\times B_{\on{Zar}}(G),\BZ_{\on{mot}}(2))$ of $X\times B_{\on{Zar}}(G)$. 
I.e., this is by definition the totalization of the cosimplicial complex 
$$\Gamma_{\on{et}}(X\times B^\bullet(G),\BZ_{\on{mot}}(2)),$$ where
for a smooth scheme $Y$ we denote by $\Gamma_{\on{et}}(Y,\BZ_{\on{mot}}(n))$ the \'etale cohomology of $\BZ_{\on{mot}}(n)$. 

\sssec{} 

Let $(V,E,(BG)^{\on{appr}})$ be as in \secref{sss:E and V}. Then 
\begin{multline*}
\Gamma_{\on{et}}(X\times B_{\on{Zar}}(G),\BZ_{\on{mot}}(2))\simeq \Gamma_{\on{et}}(X\times (V/G)_{\on{Zar}},\BZ_{\on{mot}}(2))\simeq  \\
\simeq \Gamma_{\on{et}}(X\times (E/G)_{\on{Zar}},\BZ_{\on{mot}}(2))\simeq \Gamma_{\on{et}}(X\times (E/G)_{\on{et}},\BZ_{\on{mot}}(2))\simeq 
\Gamma_{\on{et}}(X\times (BG)^{\on{appr}},\BZ_{\on{mot}}(2)).
\end{multline*} 

\sssec{}

We will now reproduce another result from \cite[Lemma 6.2]{EKLV}:

\begin{thm} \label{t:mot etale}
The map $$\Gamma(X\times B_{\on{Zar}}(G);X\times \on{pt},\BZ_{\on{mot}}(2))\to 
\Gamma_{\on{et}}(X\times B_{\on{Zar}}(G);X\times \on{pt},\BZ_{\on{mot}}(2))$$
is an isomorphism in degrees $\leq 4$.
\end{thm} 

\sssec{}

\thmref{t:mot etale} applies more generally than the pair $(X\times B_{\on{Zar}}(G),X\times \on{pt})$. Namely, we can replace 
$X\times B_{\on{Zar}}(G)$ by a prestack $\CY$ that can be realized as
the geometric realization of a smooth simplicial prestack $Y^\bullet$ over $X$, such that $Y^0\simeq X$
(in particular, the inclusion of $0$-simplices provides a map $X\to \CY$). 

\begin{proof}[Proof of \thmref{t:mot etale}]

Let $\nu_*$ denote the direct image functor from the \'etale to the Zariski site for a given scheme $Y$. We have
$$\Gamma_{\on{et}}(Y,\BZ_{\on{mot}}(2))\simeq \Gamma_{\on{Zar}}(Y,\nu_*(\BZ_{\on{mot}}(2))).$$

Hence, it suffices to show that the totalization of the cosimplicial complex that attaches to $j$
\begin{equation} \label{e:cosimpl trunc}
\Gamma_{\on{Zar}}\biggl(Y^j;X,\on{coFib}\left(\BZ_{\on{mot}}(2)\to \nu_*(\BZ_{\on{mot}}(2))\right)\biggr)
\end{equation}
lives in cohomological degrees $>4$.

\medskip

The key fact that we will use is that for a smooth scheme $Y$, the object
$$\on{coFib}(\BZ_{\on{mot}}(2)\to \nu_*(\BZ_{\on{mot}}(2)))\in \on{Shv}_{\on{Zar}}(Y,\on{Ab})$$
lives in cohomological degrees $\geq 4$. See \cite[Formula (6.3)]{EKLV} and references therein. 

\medskip

Hence, \eqref{e:cosimpl trunc} a priori lives in degrees $\geq 4$.  

\medskip

Furthermore, its 4th cohomology injects into the 4th cohomology of the $0$-simplices, while the latter vanishes
by the assumption that $Y^0\simeq X$. 

\end{proof} 

\begin{rem}
\thmref{t:mot etale} (combined with \thmref{t:mot}) 
can be seen as ``explaining" the computation of the homotopy groups of $\Maps_{\on{based}}(X\times B_{\on{Zar}}(T),B^2_{\on{Zar}}(\bK_2))$,
given by \corref{c:BD 3}. Namely, we will now show that for $\ell$ coprime with $\on{char}(k)$, 
the map \eqref{e:from K to mu mod l again} factors as 
\begin{multline*} 
\Maps_{\on{based}}(X\times B_{\on{Zar}}(T),B^2_{\on{Zar}}(\bK_2))/\ell \to  \Maps_{\on{based}}(X\times B_{\on{Zar}}(G),B^4_{\on{et}}(\BZ_\ell(2)))/\ell \hookrightarrow \\
\hookrightarrow \Maps_{\on{based}}(X\times B_{\on{Zar}}(G),B^4_{\on{et}}(\mu^{\otimes 2}_\ell)),
\end{multline*} 
where the first arrow is fully faithful (but combined with \corref{c:BD 3} we actually know that it is an equivalence). Indeed, we have: 
\begin{multline*} 
\pi_i\biggl(\Maps_{\on{based}}(X\times B_{\on{Zar}}(T),B^2_{\on{Zar}}(\bK_2))/\ell\biggr)=
H^{2-i}\biggl(\tau^{\leq 2}(\CK_2(X\times B_{\on{Zar}}(G)))/\ell\biggr)\simeq \\
\simeq 
H^{4-i}\biggl(\tau^{\leq 4}(\Gamma(X\times B_{\on{Zar}}(G);X\times \on{pt},\BZ_{\on{mot}}(2)))/\ell\biggr)
\simeq H^{4-i}\biggl(\tau^{\leq 4}(\Gamma_{\on{et}}(X\times B_{\on{Zar}}(G);X\times \on{pt},\BZ_{\on{mot}}(2)))/\ell\biggr).
\end{multline*} 

We claim now that the natural map
\begin{multline*} 
H^{4-i}\biggl(\tau^{\leq 4}(\Gamma_{\on{et}}(X\times B_{\on{Zar}}(G);X\times \on{pt},\BZ_{\on{mot}}(2)))/\ell\biggr)\to
H^{4-i}\biggl(\tau^{\leq 4}(\Gamma_{\on{et}}(X\times B_{\on{Zar}}(G);X\times \on{pt},\BZ_\ell(2)))/\ell\biggr) = \\
=\pi_i\left(\Maps_{\on{based}}(X\times B_{\on{Zar}}(G),B^4_{\on{et}}(\BZ_\ell(2)))/\ell \right)
\end{multline*} 
is an injection for $i=0$ and an isomorphism for $i>0$. Indeed, this follows from the long exact cohomology sequences
associated with
$$0\to \BZ_{\on{mot}}(2) \overset{\ell}\to \BZ_{\on{mot}}(2) \to \BZ_{\on{mot}}(2)/\ell\to 0 \text{ and }
0\to \BZ_\ell(2) \overset{\ell}\to \BZ_\ell(2) \to \mu_\ell^{\otimes 2}\to 0$$
and the fact that the map
$$H^i_{\on{et}}(Y,\BZ_{\on{mot}}(n)/\ell) \to H^i_{\on{et}}(Y,\mu_\ell^{\otimes n})$$
is an isomorphism for any $n$ and $\ell$ coprime with $\on{char}(k)$ over any smooth scheme $Y$ (see \cite[Theorem 10.2]{MVW}).

\end{rem}

\ssec{K-theory and motivic cohomology of the algebraic stack classifying space}  \label{ss:classifying}

\sssec{}

Let $\CY$ be a smooth algebraic stack. We do not a priori know how to define 
$$\Gamma(\CY,\BZ_{\on{mot}}(n)).$$

This is due to the fact that $\BZ_{\on{mot}}(n)$ does not have \'etale descent, so it is not clear how to justify that
we can probe $\CY$ by mapping to it only smooth test-schemes.

\sssec{}

Assume, however, that $\CY$ admits a map
$$Y\to \CY,\quad Y\in \on{Sch},$$
that is smooth, schematic and \emph{surjective in the Zariski topology}. 

\medskip

We claim that for such $\CY$, we can define $\Gamma(\CY,\BZ_{\on{mot}}(n))$. Namely, we let $\Gamma(\CY,\BZ_{\on{mot}}(n))$ be the totalization of
the cosimplicial complex 
$$\Gamma(Y^\bullet,\BZ_{\on{mot}}(n)),$$
where $Y^\bullet$ is the \v{C}ech nerve of $Y\to \CY$. It is easy to see that this definition is canonically independent of the
choice of $Y\to \CY$.

\sssec{Example}

Any quotient stack $(Z/G)_{\on{et}}$, where $Z$ is a smooth scheme and $G$ a linear algebraic group, 
has this property: indeed, embed $G$ into $GL_n$ and write
$$(Z/G)_{\on{et}}\simeq ((Z\times (GL_n/G)_{\on{et}})/GL_n)_{\on{et}}\simeq ((Z\times (GL_n/G)_{\on{et}})/GL_n)_{\on{Zar}},$$
so the required map is given by $Z\times (GL_n/G)_{\on{et}}\to (Z/G)_{\on{et}}$. 

\sssec{}

Let $\CY$ be of the form  $(Z/G)_{\on{et}}$, and consider a triple $(V,E,(BG)^{\on{appr}})$, where $\on{codim}(V-E,V)>n$. 
Note that the action of $G$ on $Z\times E$ is free (i.e., the algebraic stack quotient is a scheme) and denote
$$(Z/G)^{\on{appr}}:=((Z\times E)/G)_{\on{et}}.$$

We claim:

\medskip

\begin{prop}  \label{p:classifying}
The canonical maps
$$\Gamma((Z/G)_{\on{et}},\BZ_{\on{mot}}(n)) \to \Gamma((Z/G)^{\on{appr}},\BZ_{\on{mot}}(n)) \text{ and }
\CK_n((Z/G)_{\on{et}}) \to \CK_n((Z/G)^{\on{appr}})$$
are isomorphisms.
\end{prop}

\begin{proof}

We will give a proof for $\CK_n$, as the proof for $\BZ_{\on{mot}}(n)$ is similar. We need to show that the maps
$$\CK_n((Z/G)_{\on{et}})\to \CK_n(((Z\times V)/G)_{\on{et}}) \to \CK_n(((Z\times E)/G)_{\on{et}})$$
are isomorphisms. 

\medskip

Embedding $G$ into $GL_n$ and denoting $W:=(GL_n/G)_{\on{et}}$, we rewrite the above as
$$\CK_n\biggl(((Z\times W)/GL_n)_{\on{Zar}}\biggr) \to 
\CK_n\biggl(((Z\times V\times W)/GL_n)_{\on{Zar}}\biggr) \to \CK_n\biggl(((Z\times E\times W)/GL_n)_{\on{Zar}}\biggr),$$
which comes from a map of cosimplicial complexes that attaches to $j$
$$\CK_n(Z\times W\times (GL_n)^{j-1}) \to \CK_n(Z\times W\times V\times (GL_n)^{j-1})\to  
\CK_n(Z\times W\times E\times (GL_n)^{j-1}).$$

Now, the first map is a (simplex-wise) isomorphism by $\BA^1$-invariance, and the second map is a 
(simplex-wise) isomorphism by the assumption on the codimension. 

\end{proof} 

\begin{rem}   \label{r:final}

To summarize, we obtain the following commutative diagram 
$$
\CD
& & H^4(BG_{\on{top}},\BZ) \\
& &  @A{\sim}AA  \\
& & \on{Quad}(\Lambda_G,\BZ)^{W_G} \\
& &  @A{\sim}AA  \\
H^4_{\on{et}}(B_{\on{Zar}}(G),\BZ_{\on{mot}}(2))   @<{\sim}<<  H^4_{\on{Zar}}(B_{\on{Zar}}(G),\BZ_{\on{mot}}(2))   @>{\sim}>>  H^2(\CK_2(B_{\on{Zar}}(G))) \\
@A{\sim}AA   @AAA   @AAA  \\
H^4_{\on{et}}(B_{\on{et}}(G),\BZ_{\on{mot}}(2))   @<<<  H^4_{\on{Zar}}(B_{\on{et}}(G),\BZ_{\on{mot}}(2))  @>{\sim}>> H^2(\CK_2(B_{\on{et}}(G))) \\
@V{\sim}VV   @V{\sim}VV   @V{\sim}VV  \\
H^4_{\on{et}}((BG)^{\on{appr}},\BZ_{\on{mot}}(2))   @<<<  H^4_{\on{Zar}}((BG)^{\on{appr}},\BZ_{\on{mot}}(2))  @>{\sim}>> H^2(\CK_2((BG)^{\on{appr}})) \\
& & @V{\sim}V{\text{(according to the definition of \cite{To})}}V \\
& & \on{CH}_2(BG). 
\endCD
$$

The arrows \emph{not} marked as isomorphisms are \emph{not} isomorphisms in general, but correspond to the embedding of a 
subgroup of finite index. The corresponding subgroup in $H^4(BG_{\on{top}},\BZ)$ is one generated by the Chern classes of representations. 

\end{rem}

\end{document}